\documentclass[12pt]{amsart}
\usepackage[mathscr]{eucal}
\usepackage{amsmath,amsfonts}
\usepackage[all]{xy}
\usepackage{color}
\usepackage{hyperref}

\parskip=\smallskipamount

\newtheorem{theorem}{Theorem}[section]
\newtheorem{proposition}[theorem]{Proposition}

\newtheorem{lemma}[theorem]{Lemma}

\theoremstyle{definition}

\newtheorem{example}[theorem]{Example}

\newtheorem{remark}[theorem]{Remark}

\makeatletter
\@addtoreset{equation}{section}
\makeatother

\newcommand{\CC}{{\mathbb C}}
\newcommand{\NN}{{\mathbb N}}

\newcommand{\ZZ}{{\mathbb Z}}
\newcommand{\DD}{{\mathbb D}}
\newcommand{\RR}{{\mathbb R}}
\newcommand{\TT}{{\mathbb T}}

\newcommand{\cB}{{\mathcal B}}
\newcommand{\cC}{{\mathcal C}}
\newcommand{\cD}{{\mathcal D}}

\newcommand{\cF}{{\mathcal F}}
\newcommand{\cG}{{\mathcal G}}
\newcommand{\cH}{{\mathcal H}}

\newcommand{\cL}{{\mathcal L}}

\newcommand{\cP}{{\mathcal P}}
\newcommand{\cR}{{\mathcal R}}
\newcommand{\cS}{{\mathcal S}}

\newcommand{\cW}{{\mathcal W}}

\newcommand{\fA}{{\mathfrak A}}

\newcommand{\dom}{\operatorname{Dom}}

\newcommand{\ran}{\operatorname{Ran}}

\newcommand{\ra}{\rightarrow}

\let\phi=\varphi
\newcommand{\iac}{\mathrm{i}}
\renewcommand{\ker}{\operatorname{Null}}
\newcommand{\de}{\operatorname{d}}

\newcommand{\essinf}{\operatorname*{ess\,inf}}
\newcommand{\essran}{\operatorname{ess\,ran}}

\newcommand{\nr}[1]{\vspace{0.1ex}\noindent\hspace*{12mm}\llap{\textup{(#1)}}}

\title{Triplets of Closely Embedded Hilbert Spaces}\thanks{The first named author
acknowledges financial support from  the Polish Ministry of Science and
Higher Education: 11.11.420.04 and Grant NN201 546438 (2010-2013).}
\thanks{The second named
author acknowledges financial support from the 
grant of the Romanian 
National Authority for Scientific Research, CNCS Ð UEFISCDI, project number
PN-II-ID-PCE-2011-3-0119.}

\author{Petru Cojuhari}
\address{Department of Applied Mathematics, AGH University of Science and
Technology,\hfill\break Al.~Mickievicza 30, 30-059 Cracow, Poland}
\email{cojuhari@uci.agh.edu.pl}

\author{Aurelian Gheondea}
\address{Department of Mathematics, Bilkent University, 06800 Bilkent, Ankara,
  Turkey, \emph{and} Institutul de Matematic\u a al Academiei Rom\^ane, C.P.\
  1-764, 014700 Bucure\c sti, Rom\^ania}
\email{aurelian@fen.bilkent.edu.tr \textrm{and} A.Gheondea@imar.ro}

\begin{document}

\begin{abstract}  We obtain a general concept of triplet of Hilbert spaces with 
closed (unbounded) embeddings instead of continuous (bounded) ones. 
We provide a model and an abstract theorem as well for a triplet of closely 
embedded Hilbert spaces associated to positive selfadjoint operator $H$, that 
is called the Hamiltonian of the system, which is supposed to be 
one-to-one but may not have a bounded inverse. 
Existence and uniqueness results, as well as left-right symmetry, 
for these triplets of closely embedded Hilbert spaces are obtained.
 We motivate this abstract theory by a diversity of problems 
coming from homogeneous or weighted Sobolev spaces, Hilbert spaces of 
holomorphic functions, and weighted $L^2$ spaces. 
An application to weak solutions 
for a Dirichlet problem associated to a class of degenerate elliptic partial 
differential equations is presented. In this way, we propose a 
general method of proving the existence of weak solutions that avoids 
coercivity conditions and Poincar\'e-Sobolev type inequalities.
\end{abstract}

\maketitle

\section{Introduction}

The concept of rigged Hilbert space was introduced and investigated by 
I.M.~Gelfand and A.G.~Kostyuchenko \cite{GelfandKostyuchenko}, see
\cite{GelfandVilenkin} for further developments, in connection to the general 
problem of reconciliating the two basic paradigms of Quantum 
Mechanics, that of P.A.M.~Dirac based on bras 
and kets and used mainly by physicists, with that of J.~von Neumann based on 
positive selfadjoint operators in
Hilbert spaces and used mainly by mathematicians. 
This reconciliation was essentially
facilitated by the L.~Schwartz's theory of distributions \cite{Schwartz}. 
Briefly, a rigged Hilbert space
is a triplet $(\cS;\cH;\cS^*)$, in which $\cH$ is a complex Hilbert space, $\cS$ is a 
topological vector space that is continuously and densely embedded in $\cH$, 
while $\cS^*$
is the "dual space of $\cS$ with respect to $\cH$" and such that $\cH$ is 
continuously and densely embedded in $\cS^*$. The rigged Hilbert space 
formalism was later recognized and used by physicists as a powerful 
and rigorously mathematical tool for problems in quantum mechanics, 
e.g.\ see A.~Bohm and M.~Gadella \cite{BohmGadella}, R.~de~la~Madrid 
\cite{Madrid}, and the rich bibliography cited there. In particular, a theory, consistent 
both mathematically and physically, of Gamow states and of 
quantum resonances was made possible, e.g.~see the survey article of 
O.~Civitarese and M.~Gadella \cite{CivitareseGadella}.

One of the built-in deficiency of the theory of rigged Hilbert spaces consists on
the vague formalization of the meaning of "dual space of $\cS$ with respect to 
$\cH$". In this respect,
an important contribution to the theory of rigged Hilbert spaces is due to
Yu.M. Berezansky \cite{Berezanski}, \cite{Berezanskii}, and his school \cite{BerezanskyKondratiev}, 
\cite{BerezanskySheftelUs},  in which rigged Hilbert spaces are generated by 
scales of continuously embedded Hilbert spaces 
with certain properties. The basic concept in this 
approach is that of a \emph{triplet of Hilbert spaces}.
More precisely, this is denoted by 
$\cH_{+}\hookrightarrow \cH\hookrightarrow \cH_{-}$, 
where: $\cH_{+}$, $\cH_0$, and $\cH_-$ are 
Hilbert spaces, the embeddings are continuous (bounded linear operators), the 
space $\cH_+$ is dense in $\cH_0$, the space $\cH_0$ is dense in $\cH_-$, 
and the space $\cH_-$ is the conjugate dual of $\cH_+$ with respect to $\cH_0$, 
that is, $\|\phi\|_-=\sup\{ |\langle h,\phi
\rangle_{\cH}\mid \|h\|_+\leq 1\}$, for all $\phi\in \cH_0$.
Extending these triplets on 
both sides, one gets a \emph{scale of Hilbert spaces} 
that yields, by an inductive limit and, respectively, a projective limit, a 
\emph{rigged Hilbert space} $\cS\hookrightarrow \cH\hookrightarrow \cS^*$. 
In this respect, the rather vague notion of "duality through a Hilbert space" is 
made precise, as well.

In order to produce a triplet of Hilbert spaces, this method requires that
the positive selfadjoint operator which generates it, 
and that we call the  \emph{Hamiltonian} of the system, should have a 
bounded inverse. In the following we briefly describe this construction, 
following \cite{Berezanskii} 
and \cite{BerezanskySheftelUs}, but with different notation and making explicit 
a technique of operator ranges, e.g.\ see \cite{FlWl} and the rich bibliography 
cited there.
Let $H$ be a positive selfadjoint operator in a Hilbert space
$\cH$ such that $A=H^{-1}$ is a bounded operator. Then there exists 
$S\in\cB(\cH)$ such that $A=S^*S$, e.g.\ $S=A^{1/2}$ 
does the job. Note that, necessarily, $S$ has trivial kernel and dense range, but 
may not be boundedly invertible. Let $\cR(S)$ denote the range space $\ran(S)$, 
hence a dense linear manifold in $\cH$, organized as a Hilbert space with respect 
to the norm
\begin{equation}\label{e:snorm} \|f\|_S=\|u\|_\cH,\quad f=Su,\ u\in\cH.
\end{equation}
Then $\cH_+=\cR(S)$ is continuously embedded in $\cH$, let $j_+$ denote this 
embedding, and note that $j_+j_+^*=A$, the \textit{kernel operator} of 
this embedding. 

On $\cH$ one can define a new norm $\|\cdot\|_-$ by the variational 
formula
\begin{equation}\label{e:minusnorm} \|f\|_-=\sup\{\frac{|\langle f,u\rangle_\cH|}{\|u\|_+}\mid u
\in \cH_+\setminus \{0\}\},
\end{equation} and let $\cH_-$ denote the completion of $\cH$ under the norm 
$\|\cdot\|_-$. Then $\cH$ is continuously embedded and dense in $\cH_-$; 
let $j_-$ denote the bounded operator of embedding $\cH$ into $\cH_-$. Thus,
$(\cH_+;\cH;\cH_-)$ is a triplet of Hilbert spaces. The following theorem gathers a 
few remarkable facts, cf.\ \cite{Berezanskii} and
\cite{BerezanskySheftelUs}.

\begin{theorem}\label{t:berezansky} Let $H$ be a positive selfadjoint operator 
in a Hilbert space
$\cH$ such that $A=H^{-1}$ is a bounded operator, and let 
$S\in\cB(\cH)$ be such that $A=S^*S$. With notation as before $(\cH_+;\cH;\cH_-)$ is a triplet of Hilber spaces. In addition:
\begin{itemize}
\item[(a)] The operator $j_+^*\colon \cH\ra \cH_+$, when viewed as an operator 
densely defined in $\cH_-$ and valued in $\cH_+$, 
can be uniquely extended to a unitary operator $\widetilde V\colon \cH_-\ra\cH_+$.
\item[(b)] The kernel operator $A$ can be viewed as a linear operator densely 
defined in $\cH_-$, with dense range in $\cH_+$, and it is a restriction of 
the unitary operator $\widetilde V$, as in item \emph{(a)}.
\item[(c)] The Hamiltonian operator $H$ can be viewed as an operator 
densely defined in $\cH_+$ and valued in $\cH_-$, and then it has a unique 
unitary extension $\widetilde H\colon \cH_+\ra\cH_-$ such that $\widetilde H=\widetilde V^{-1}$.
\item[(d)] The operator $\Theta\colon \cH_-\ra\cH_+^*$ (here $\cH_+^*$ denotes 
the conjugate dual space 
of $\cH_+$), defined by $(\Theta y)(x)=\langle \widetilde V y,x\rangle_+$, for 
$y\in \cH_-$ and $x\in\cH_+$, provides the canonical identification of $\cH_-$ 
with $\cH_+^*$.
\end{itemize}
\end{theorem}

One of the most important applications of 
Theorem~\ref{t:berezansky} is to the method of weak 
solutions for boundary value problems associated to certain 
partial differential equations. The assumption in Theorem~\ref{t:berezansky} 
that the operator $H$  has a bounded inverse requires, in terms of the 
corresponding boundary value problem, the 
Lax-Milgram Theorem referring to a 
bilinear form that is bounded away from zero, the so-called coercivity 
condition, that is usually proven by means of subtle Poincar\'e-Sobolev type 
inequalities, which can be rather technical and restricting very much the range 
of applications, e.g.\ see L.C.~Evans \cite{Evans}, E.~Sanchez-Palencia
\cite{Sanchez}, or R.E.~Showalter \cite{Showalter}. 
Our point of view, as illustrated by the main results Theorem~\ref{t:model} and 
Theorem~\ref{t:triplet}, is that this technical condition can be weakened by 
means of the more general concept of triplets of closely embedded Hilbert 
spaces that we propose herewith. In order to substantiate this, we provide in 
Section~\ref{s:wsg} an 
application of our main results to provide existence of  weak solutions for 
Dirichlet problems associated to degenerate elliptic operators. 

In Section~\ref{s:motivation} 
we show that there are strong motivations, coming from problems
related to
homogeneous Sobolev spaces, weighted Sobolev spaces, Hilbert spaces of 
holomorphic functions, weighted $L^2$ spaces, and others, that require 
dropping the assumption that the Hamiltonian operator $H$ admits a bounded inverse.
In Section~\ref{s:tcehs} we show that a sufficiently rich and consistent theory 
for triplets of Hilbert 
spaces can be obtained by replacing the notion of continuous
embedding by that of a closed embedding, cf.\ \cite{CojGh3}, within a more general
concept of triplet of closely embedded Hilbert spaces. More precisely, 
by employing this new concept of triplets of closely embedded Hilbert 
spaces, in Theorem~\ref{t:triplet} we 
essentially recover all of the properties  (a)--(d) from 
Theorem~\ref{t:berezansky} in the more general case when the 
Hamiltonian is free of any coercivity assumption and, in this way, providing
an approach to the motivating problems listed before.

In order to single out the concept of a triplet of closely embedded Hilbert 
spaces we make use of our previous investigations on closed embeddings in 
\cite{CojGh3}. The correct axioms of a triplet of closely embedded Hilbert 
spaces became clearer to us first as a consequence of a "test of validity" of 
this model  on Dirichlet type spaces on the unit polydisc as in \cite{CojGh4} 
and, secondly, as an abstract model generated by an arbitrary 
factorization $H=T^*T$ of the Hamiltonian operator, that we obtain in 
Section~\ref{s:model}.

\section{Some Motivations} \label{s:motivation}

In this section we record a few of the problems that lead us to considering 
generalizations of triplets of Hilbert spaces. 

\subsection{Bessel Potential versus Riesz Potential.}\label{ss:bessel} 
We first point out a triplet of Hilbert spaces associated to continuous embeddings
of some Sobolev Hilbert spaces in $L_2(\mathbb{R}^n)$, following the 
Remark~4.3 in \cite{CojGh3}. We assume the reader to be familiar 
with the basic 
terminology and facts on various Sobolev spaces as presented, e.g.\ in the 
monographs of R.A.~Adams \cite{Adams}, V.M.~Maz'ja \cite{Mazja}, 
S.L.~Sobolev \cite{Sobolev}, or 
R.A.~Adams and J.~Fournier \cite{AdamsFournier}. A few notation is recalled
in Section~\ref{s:wsg}.

Let $\mathcal{H} = L_{2} (\mathbb{R}^{n}), \ \ n \geq 3,$
and let $H_1$ denote the operator $H_1 = (-\Delta + I)^{l}$,
where $\Delta \equiv \sum_{k = 1}^{n} \partial^{2}/\partial
x_{k}^{2}$  is the Laplacian and $l$ is a positive number. For the case when $l$ is integer, see Section~\ref{s:wsg} for notation.
  As the domain of $H_1$, the Sobolev
space $W_{2}^{\alpha} (\mathbb{R}^{n})$, $\alpha = 2 l$ is considered.
$H_1$ represents on this domain a positive definite selfadjoint
operator. In particular, $H$ is an invertible operator, and its
inverse is bounded on $\mathcal{H}$. Next, we denote
\begin{equation*}S = (-\Delta + I)^{-l/2}.\end{equation*}

The operator $S$ can be represented, e.g.\ see E.M.~Stein \cite{St}, \S
V.3.1, as a convolution integral
operator with  kernel
\begin{equation*}
G (x) = c K_{(n - l)/2} \ (| x |) | x |^{(l - n)/2},\end{equation*}
where $K_{\nu}$ is the modified Bessel function of the third kind, $c$
is a positive constant, see e.g.\
N.~Aronszajn and K.T.~Smith \cite{AS}, \S~II.3.
Thus
\begin{equation*}
(S u) (x) = \int_{\mathbb{R}^{n}}  G_{l} (x - y) u (y) \de y, \ \ \ u \in
L_{2} (\mathbb{R}^{n}).\end{equation*}
This integral operator is known as \emph{the Bessel potential of order} $l$,
e.g.\ see \cite{St}.

Note that $S$ can be also regarded as a pseudodifferential operator
corresponding to the symbol $(1 + | \xi |^{2})^{-l/2}$,  i.e.\
\begin{equation*}
(S u) (x) = \frac{1}{(2 \pi)^{n/2}}  \int_{\mathbb{R}^{n}}  (1 + | \xi
|^{2})^{-l/2} \widehat{u}
(\xi) e^{-\iac \langle x,\xi \rangle} \de \xi, \quad x \in
\mathbb{R}^{n},\end{equation*} where $\widehat{u} = \cF u$  is the Fourier
transform of the function $u \in L_{2} (\mathbb{R}^{n})$ and
$\langle x, \xi \rangle$ denotes the scalar product of the
elements $x,\xi \in \mathbb{R}^{n}$.
Obviously, $S$ maps $L_{2} (\mathbb{R}^{n})$ onto $W_{2}^{l} (\mathbb{R}^{n})$.

Following \eqref{e:snorm} we define an inner product on 
$\ran (S) \ ( = W_{2}^{l}(\mathbb{R}^{n}))$  by setting
\begin{equation*}
\langle S f, S g \rangle_{S} : = \langle f,g \rangle_{\mathcal{H}},
\quad f,g \in \mathcal{H}.\end{equation*}
We have
\begin{equation*}
\langle u,v \rangle_{S} = \langle (-\Delta + I)^{l/2} u, (-\Delta + I)^{l/2}
v \rangle_{\mathcal{H}},
 \ \ \ u,v \in \ran (S),\end{equation*}
 and, respectively, for the corresponding norm
 \begin{equation*}\| u \|_{S} = \| (-\Delta + I)^{l/2} u
 \|_{\mathcal{H}}, \quad u \in \ran (S).\end{equation*}
 This norm is equivalent with the standard norm
 \begin{equation*}\| u \|_{W_{2}^{l} (\mathbb{R}^{n})} : = \biggl(
 \int_{\mathbb{R}^{n}} | u (x) |^{2} \de x +
  \int_{\mathbb{R}^{n}}  | (\nabla_l u ) (x) |^{2} \de x
  \biggl)^{1/2}\end{equation*}
of the Sobolev space $W_{2}^{l} (\mathbb{R}^{n})$.  Consequently,
$\ran (S)$ endowed with the norm $\| \cdot \|_{S}$ coincides with
the Sobolev space $W_{2}^{l} (\mathbb{R}^{n})$. Thus $\mathcal{R}
(S) = W_{2}^{l} (\mathbb{R}^{n})$ algebraically and topologically.  
Moreover, $\mathcal{R} (S)$ is
continuously embedded in $\mathcal{H}$ and
the kernel operator of the canonical embedding is the Bessel
potential $J_{\alpha} = (- \Delta + I)^{-\alpha/2}$ of order
$\alpha = 2 \ell$.  Note that
  \begin{equation*}
\langle u,v \rangle_{S} = \langle H u,v \rangle_{\mathcal{H}}, \quad
  u \in \dom (H), \ v \in \mathcal{R} (S).\end{equation*}
We can now apply Theorem~\ref{t:berezansky} and get a  triplet of Hilbert spaces 
$(W_2^l(\mathbb{R}^n),L_2(\mathbb{R}^n),W_2^{-l}(\mathbb{R}^n))$, where 
$W_2^{-l}(\mathbb{R}^n)$ denotes the conjugate dual space of 
$W_2^l(\mathbb{R}^n)$, and the Hamiltonian operator is $H_1=(-\Delta + I)^{l}$. 

Let now $H^l_2(\RR^n)$ denote
the homogeneous Sobolev space of all functions
    $u \in W_{2, \mathrm{loc}}^{l} (\mathbb{R}^{n})$  for which
  $ \| u \|_{2,l}^{2}  < \infty$, where
  \begin{equation}\label{e:five}
\| u \|_{2,l}^{2} : = \int_{\mathbb{R}^{n}} (| ( \nabla_{l} u) (x)
|^{2} + | x |^{-2l} | u (x) |^{2}) \de x ,
 \quad  u \in C_{0}^{\infty} ({\mathbb{R}^{n}}).
\end{equation}
The operator $H_0=(-\Delta)^{l}$ is
defined on its maximal domain, i.e.\ on the Sobolev space
$W_{2}^{\alpha} (\mathbb{R}^{n})$, $\alpha = 2 l$, and it represents a selfadjoint
operator in $\mathcal{H}$. When trying to perform a similar treatment as in the case 
corresponding to the operator $H_1=(\Delta+I)^l$ and described before, 
it turns out that the 
Hamiltonian operator $H_0$ is one-to-one but it does not have a bounded inverse. 
Instead of the Bessel potential that yields a bounded integral operator, we get
the Riesz potential that yields an unbounded integral operator. In Subsection~4.2 in
\cite{CojGh3} we described a way of treating this case by means of "closely 
embedding" the homogeneous Sobolev space $H^l_2(\RR^n)$ into $L_2(\RR^n)$, 
which is actually associated to the Hamiltonian $H_0$ and cannot be continuously
embedded in $L_2(\RR^n)$, and
which, once again, makes a motivation for changing the definition of
the triplet of Hilbert spaces with a more general one.

More precisely, we consider the operator $T$ defined in the space $L_{2}
(\mathbb{R}^{n})$  by
\begin{equation*}
(T u) (x) = \frac{1}{(2 \pi)^{n/2}} \int_{\mathbb{R}^{n}}  | \xi
|^{-l/2}  \widehat{u} (\xi)
e^{-\iac \langle x,\xi \rangle}  \de \xi, \quad x \in \mathbb{R}^{n},\end{equation*}
on the domain
\begin{equation*}\dom (T) : = \{u \in L_{2} (\mathbb{R}^{n}) \mid  | \xi
 |^{-l/2} \widehat{u} (\xi) \in L_{2} (\mathbb{R}^{n})\}.\end{equation*}
 The operator $T$ can be written formally as
 \begin{equation*}T = (-\Delta)^{-l/2},\end{equation*}
 and it can be also considered as \emph{the M.~Riesz potential of order} $l$,
 e.g.\ see E.M.~Stein \cite{St}, \S~V.1.1,
that means that $T$ is the convolution integral operator
 with the kernel $|x|^{l - n}$, up to a constant,
 \begin{equation*}
(T u) (x) = c \int_{\mathbb{R}^{n}}
 \frac{u (y)}{| x - y |^{n - l}}
 \de y, \quad u \in \dom (T).\end{equation*}
 $T$ represents a closed unbounded operator in
 $\mathcal{H}$ $( = L_{2} (\mathbb{R}^{n}))$, and, obviously,
 $\ker (T) = \{ 0 \}.$   The domain of $T$ is $\ran (H_0^{1/2})$  and its
 range  is $\dom (H_0^{1/2}),$ i.e. the Sobolev space $W_{2}^{l} (\mathbb{R}^{n})$.
In Theorem~4.4 in \cite{CojGh3} it is proven that, by employing the more general
notion of "closed embedding" and providing the necessary generalization of the 
"operator range" space $\cR(T)$, see Subsection~\ref{ss:thsrt}, 
one can prove that the homogeneous Sobolev space $H^l_2(\RR^n)=\cR(T)$.

\subsection{Weighted Sobolev Spaces.} Let $\Omega$ be a domain (nonempty 
open set) in $\RR^N$. A \textit{weight} $w$ on $\Omega$ is a measurable function  
$\omega\colon \Omega\ra (0,+\infty)$. In this case, the weighted Hilbert space 
$L^2_w(\Omega)$ consists of all measurable functions $f\colon\Omega\ra\CC$ 
such that 
\begin{equation}\label{e:2wnorm}\|f\|_{2,w}^2
=\int_\Omega |f(x)|^2w(x)\de x<+\infty.\end{equation}
Following A.~Kufner and B.~Opic \cite{KufnerOpic}, 
a weight $w$ on $\Omega$ satisfies \emph{condition} $B_2(\Omega)$ if 
$w^{-1}\in L^1_{\mathrm{loc}}(\Omega)$. An application of Schwarz Inequality
shows that, if the weight $w$ satisfies condition $B_2(\Omega)$, then 
$L^2_w(\Omega)$ is continuously embedded in $L^1_{\mathrm{loc}}(\Omega)$,
in particular $L^2_w(\Omega)\subset \cD^\prime(\Omega)$, the space 
of distributions on $\Omega$ and hence, for every function 
$u\in L^2_w(\Omega)$ and multi-index $\alpha\in \NN_0^N$, the distributional 
derivatives $D^\alpha u$ make sense.

Letting $\cW=\{w_j\}_{j=0}^N$ be a family of weights on $\Omega$,
for any $u\in L^2_{w_0}(\Omega)\cap L^1_\mathrm{loc}(\Omega)$ such that for 
$j=1,\ldots,N$ the distributional derivatives $\partial u/\partial x_j$ are regular
distributions associated to functions in $L^2_{w_0}(\Omega)\cap 
L^1_\mathrm{loc}(\Omega)$, one can define the norm
\begin{equation}\label{e:w12norm} \|u\|_{2,\cW}
=\bigl(\sum_{j=0}^N \|\partial u/\partial x_j\|_{2,w_j}^2\bigr)^{1/2}.
\end{equation}
If $W^1_2(\Omega;\cW)$ defines the weighted Sobolev 
space of all functions $u$ as before, 
endowed with the norm \eqref{e:w12norm}, and assuming that all weights 
$w_j$, for 
$j=1,\ldots,N$ belong to the class $B_2(\Omega)$, then $W^1_2(\Omega;\cW)$
is a Banach space, cf.\ Theorem 2.1 in \cite{KufnerOpic}. However, as proven in 
Example 1.12 in \cite{KufnerOpic}, if $\Omega=(-1,1)$, $w_0(x)=x^2$, and 
$w_1(x)=x^4$, then $W^1_2(\Omega;\cW)$, with $\cW=\{w_0,w_1\}$, 
is not complete 
with respect to the norm \eqref{e:w12norm}.

Because of the anomaly in the definition of the weighted Sobolev spaces 
$W^1_2(\Omega;\cW)$ described before, A.~Kufner and B.~Opic proposed in 
\cite{KufnerOpic} to remove the "exceptional sets" $M_2(w_j)$ for all 
$j=1,\ldots,N$, where, for a given weight $w$ on $\Omega$, they defined
\begin{equation}\label{e:em2w} 
M_2(w)=\{x\in \Omega \mid \int_{\Omega\cap U(x)} w^{-1}(y)\de y=\infty
\mbox{ for all neighbourhoods } U(x)\mbox{ of }x\}.
\end{equation}
As proven in Theorem 3.3 in \cite{KufnerOpic}, if a weight $w$ is continuous a.e.\
on $\Omega$, then the exceptional set $M_2(w)$ has Lebesgue measure zero. 
However, there are situations when this set can be rather large, or even the 
whole $\Omega$.

\begin{example}\label{e:tekin} This example was obtained by \"O.F.~Tekin
as a Senior Project under the supervision of the second named author, during the 
Fall semester of 2011, \cite{Tekin}. Let $\Omega=(0,1)$ for $N=1$ and define
\begin{equation*} w^{-1}(x)=\sum_{(m,n):\frac{m}{2^n}>x} 
\frac{1}{(\frac{m}{2^n}-x)2^{3n}}, \quad x\in (0,1),
\end{equation*} more precisely, for each $x\in (0,1)$, the terms are summed for all 
pairs of natural numbers $(m,n)$ such that $x<m/2^n$. Then $\omega$ is a weight 
on $(0,1)$ and the exceptional set $M_2(\Omega)=(0,1)=\Omega$.
\end{example}

These anomalies suggest that, as an alternative, one can define
the weighted Sobolev space $W^1_2(\Omega;\cW)$ as the completion, under the 
norm \eqref{e:w12norm}, of the space of all functions $u$ for which the norm 
$\|\cdot\|_{2,\cW}$ was 
originally defined. As noted in Remark 3.6 in \cite{KufnerOpic}, if this new
definition is adopted, then the 
space $W^1_2(\Omega;\cW)$ may contain nonregular distributions and also
functions whose distributional derivatives are not regular distributions, and hence
they considered this definition to be unnatural. Our point
of view is that, by considering the more general concepts of
closed embeddings and triplets of closely 
embedded Hilbert spaces, and developing a sufficiently rich theory for them, 
this latter definition of weighted Sobolev spaces may be reconsidered, at 
least in view of some usual problems in the theory of Sobolev spaces.

\subsection{Dirichlet Type Spaces on the Polydisc.}\label{ss:dirichlet}
For a fixed natural number $N$ consider
the unit polydisc $\DD^N=\DD\times \cdots\times\DD$, the direct product of $N$ 
copies 
of the unit disc $\DD=\{z\in\CC\mid |z|<1\}$. We consider $H(\DD^N)$  the algebra 
of functions holomorphic in the 
polydisc, that is, the collection of all functions $f\colon \DD^N\ra \CC$ that are 
holomorphic in each variable, equivalently, there exists $(a_k)_{k\in\ZZ_+^N}$ 
with the property that
\begin{equation}\label{e:rep} f(z)=\sum_{k\in\ZZ_+^N} a_kz^k,\quad z\in\DD^N,
\end{equation} where the series converges absolutely and uniformly on any 
compact subset in $\DD^N$. Here and in the sequel, for any multi-index 
$k=(k_1,\ldots,k_N)\in\ZZ_+^N$ and any $z=(z_1,\ldots,z_N)\in \CC^N$ we let 
$z^k=z_1^{k_1}\cdots z_N^{k_N}$.  

Let $\alpha\in\RR^N$ be fixed. Following G.D.~Taylor \cite{Taylor}, 
for the one dimensional case, and D.~Jupiter and D.~Redett \cite{JupiterRedett}, 
for the multidimensional case,
the \emph{Dirichlet type space} $\cD_\alpha$ is defined as the space of 
all functions $f\in H(\DD^N)$ with representation \eqref{e:rep} 
subject to the condition
\begin{equation}\label{e:cond} \sum_{k\in\ZZ_+^N} (k+1)^\alpha |a_k|^2<\infty,
\end{equation}
where, $(k+1)^\alpha=(k_1+1)^{\alpha_1}
\cdots(k_N+1)^{\alpha_N}$.
By Proposition~2.5 in \cite{JupiterRedett}, 
the condition \eqref{e:cond} implies that the function $f$ defined as 
in \eqref{e:rep} is holomorphic in $\DD^N$, so $\cD_\alpha$ is a 
subspace of $H(\DD^N)$ no matter whether we stipulate it in advance or not.
The linear space 
$\cD_\alpha$ is naturally organized as a Hilbert space with inner product 
$\langle\cdot,\cdot\rangle_\alpha$
\begin{equation}\label{e:ip} \langle f,g\rangle_\alpha
=\sum_{k\in\ZZ_+^N} (k+1)^\alpha 
a_k \overline{b_k},
\end{equation} where $f$ has representation \eqref{e:rep} and similarly
$g(z)=\sum_{k\in\ZZ_+^N} b_k z^k$, for all $z\in\DD^N$, 
and norm $\|\cdot\|_\alpha$ defined by
\begin{equation}\label{e:norm} \|f\|^2_\alpha=\sum_{k\in\ZZ_+^N} (k+1)^\alpha 
|a_k|^2.\end{equation}

For any $\alpha\in \RR^N$, on the polydisc $\DD^N$ the following 
kernel is defined
\begin{equation}\label{e:kalpha} K^\alpha
(w,z)=\sum_{k\in\ZZ_+^N} (k+1)^{-\alpha} \overline{w}^k z^k,\quad 
z,w\in\DD^N,
\end{equation} where, for $w=(w_1,\ldots,w_N)\in\DD^N$ one denotes 
$\overline{w}=(\overline{w}_1,\dots,\overline{w}_N)$, the entry-wise 
complex conjugate. We let 
$K^\alpha_w=K^\alpha(w,\cdot)$. It turns out, as follows from Lemma~2.8 and 
Lemma~2.9 in \cite{JupiterRedett}, that  $K^\alpha$ is  
the reproducing kernel for the space $\cD_\alpha$ in the sense that the following 
two properties hold:
\begin{itemize} 
\item[(rk1)] $K^\alpha_w\in \cD_\alpha$ for all $w\in\DD^N$.
\item[(rk2)] $f(w)=\langle f,K^\alpha_w\rangle_\alpha$ for all $f\in\cD_\alpha$ 
and all $w\in\DD^N$.
\end{itemize}
A more general argument shows, e.g.\ see N.~Aronszajn \cite{Aronszajn},
that the set $\{K^\alpha_w\mid w\in \DD^N\}$ is total in $\cD_\alpha$ and that the 
kernel $K^\alpha$ is positive semidefinite,.

A partial order relation $\geq$ on $\RR^N$ can be defined by $\alpha\geq\beta$ 
if and only if $\alpha_j\geq \beta_j$ for all $j=1,\ldots,N$. In addition, 
$\alpha>\beta$ means $\alpha_j>\beta_j$ for all $j=1,\ldots,N$.

The Dirichlet type space $\cD_0$ coincides with the Hardy space $H^2(\DD)$. More
precisely, following W.~Rudin \cite{Rudin},
let $\TT=\partial\DD$ denote the one-dimensional torus 
(the unit circle centered at $0$ in the complex plane) 
and then let $\TT^N=\TT\times \cdot\times \TT$ be the $N$-dimensional torus, 
also called \emph{the distinguished boundary} of the unit polydisc $\DD^N$, 
which is only a subset of $\partial \DD^N$. We consider the product measure 
$\de m_N=\de m_1\times \cdots\times\de m_1$ on $\DD^N$, 
where $\de m_1$ denotes the normalized Lebesgue measure on $\TT$, 
and for any function $f\in H(\DD^N)$ and $0\leq r<1$ let $f_r(z)=f(rz)$ 
for $z\in\DD^N$. By definition, $f\in H(\DD^N)$ belongs to $H^2(\DD^N)$ 
if and only if
\begin{equation*} \sup_{0\leq r<1} \int_{\TT^N} |f_r|^2\de m_N<\infty,
\end{equation*} and the norm $\|\cdot\|_0$  and inner product $\langle\cdot,\cdot
\rangle_0$ on the Hardy space $H^2(\DD^N)$ are
defined by
\begin{equation*} \|f\|_0^2=\sup_{0\leq r<1}\int_{\TT^N} |f_r|^2\de m_N =\lim_{r\ra 1-}
\int_{\TT^N} |f_r|^2\de m_N, \quad f\in H^2(\DD^N),
\end{equation*}
\begin{equation*}\langle f,g\rangle_0=\lim_{r\ra 1-}
\int_{\TT^N} f_r\overline{g_r}\de m_N, \quad f,g\in H^2(\DD^N),
\end{equation*}
where, we can use the lower index $0$ because it can be easily 
proven that this norm coincides with the norm $\|\cdot\|_0$ with
definition as in \eqref{e:norm} (here $0$ is the multi-index 
with all entries null). Thus, $\cD_0$ coincides as a Hilbert space 
with $H^2(\DD^N)$.
In addition, the reproducing kernel $K^0$ has a simple representation in this case, 
namely in the compact form
\begin{equation*} K^0(w,z)=\frac{1}{1-\overline{w}_1 z_1}\cdots 
\frac{1}{1-\overline{w}_Nz_N}.
\end{equation*}

In the following proposition 
we point out that a natural triplet of Hilbert spaces can be made
by rigging $\cD_0=H^2(\DD^N)$ when we consider multi-indices $\alpha\geq 0$.
In order to describe precisely the operators associated to the triplet, like kernel 
operators, Hamiltonian, and so on, we need a class of linear operators that are 
in the family of \emph{radial derivative operators}, cf.\ F.~Beatrous and J.~Burbea 
\cite{BeatrousBurbea}.

Let $\cP_N$ denote the complex vector space of polynomial functions in 
$N$ complex variables, that is, those functions $f$ that admit a 
representation \eqref{e:rep} 
for which $\{a_k\}_{k\in\ZZ_+^N}$ has finite support.
We consider now the additive group $\RR^N$ and a representation 
$T_\cdot \colon\RR^N\ra \cL(\cP_N)$, where $\cL(\cP_N)$ denotes the algebra of 
linear maps on the vector space $\cP_N$, defined by
\begin{equation}\label{e:tealpha} (T_\alpha f)(z)=\sum_{k\in\ZZ_+^N} 
(k+1)^\alpha a_k z^k,\quad \alpha\in\RR^N\ z\in\DD^N,
\end{equation} where the polynomial $f$ has representation \eqref{e:rep} and 
$\{a_k\}_{k\in\ZZ_+^N}$ has finite support. 

Theorem~\ref{t:berezansky} provides the abstract framework to precisely describe 
a triplet of Hilbert spaces $(\cD_\alpha;H^2(\DD^N);\cD_{-\alpha})$, when 
$\alpha\geq 0$. We record this in the following proposition, where the underlying 
spaces and operators are precisely described, for details see \cite{CojGh4}.

\begin{proposition}\label{p:dirichlet}
For any $\alpha\in\RR^N$ with $\alpha\geq 0$, 
$(\cD_\alpha;H^2(\DD^N);\cD_{-\alpha})$ is a
triplet of Hilbert spaces with the following properties:
\begin{itemize}
\item[(a)] The embeddings $j_\pm$ of $\cD_\alpha$ in $H^2(\DD^N)$ and, 
respectively, of $H^2(\DD^N)$ in $\cD_{-\alpha}$, are bounded and have dense 
ranges.
\item[(b)] The adjoint $j_+^*$ is defined by $j_+^*f=T_{-\alpha}f$ 
for all $f\in\dom(j_+^*)=H^2(\DD^N)\cap\cD_{-\alpha}$.
\item[(c)] The kernel operator $A=j_+j_+^*$ is a nonnegative bounded 
operator in the Hilbert space $H^2(\DD^N)$, defined by
$A f=T_{-\alpha}f$ for all $f\in H^2(\DD^N)$ 
and is an integral operator with kernel $K^{\alpha}$, in the sense that, for 
all $f\in H^2(\DD^N)$, we have
\begin{equation}\label{E:kernopalpha} (A f)(z)=\langle f,
{K^{\alpha}_z}\rangle_0=\lim_{r\ra1-} \int_{\TT^N} f_r(w) 
K^{\alpha}(rw,z)\de m_N(w),\quad  z\in\DD^N.
\end{equation}
\item[(d)] The Hamiltonian operator $H=A^{-1}$ is a positive selfadjoint
operator in $H^2(\DD^N)$ defined by $H f=T_{\alpha} f$ for all 
$f\in \dom(H)=H^2(\DD^N)\cap \cD_{2\alpha}$.
\item[(e)] The canonical unitary identification of $\cD_{-\alpha}$ with 
$\cD_\alpha^*$ is defined by
\begin{equation*} (\Theta g)f=\langle T_{-\alpha}f,g\rangle_\alpha,\quad 
f\in\cD_{-\alpha},\ g\in \cD_\alpha.
\end{equation*}
\end{itemize}

In addition, $\sigma(A)\setminus\{0\}=\{(k+1)^{-\alpha}\mid k\in\ZZ_+^N\}$ 
and $\sigma(H)\setminus\{0\}=\{(k+1)^{\alpha}\mid k\in\ZZ_+^N\}$. 
Moreover, if $\alpha_j>0$ for all $j=1,\ldots,N$, the kernel operator $A$ is 
Hilbert-Schmidt.
\end{proposition}

This proposition can be used to describe a rigging $(\cS(\DD^n), H^2(\DD^N),
\cS^*(\DD^N))$, by Dirichlet type spaces and Bergman type spaces, see 
\cite{CojGh4}.

Because, in this special case of the unit polydisc, the coefficients on different
directions are independent, a natural question that can be raised is what can
be said when considering a multi-index $\alpha\in\RR^N$ that contains positive as 
well as negative components, from the point of view of the triplet $(\cD_\alpha;
\cD_0;\cD_{-\alpha})$ as in Proposition~\ref{p:dirichlet}. It is clear that, in this case,
there is no continuous embedding of $\cD_\alpha$ in $\cD_0$. However, as 
proven directly in \cite{CojGh4}, the statements of Proposition~\ref{p:dirichlet} 
have natural generalizations, with very similar transcription, in terms of unbounded 
operators. This transcription, with appropriate definitions of closed embeddings 
and triplets of closed embeddings of Hilbert spaces, has been obtained directly
in \cite{CojGh4} because of the relative tractability of the problem, but an abstract
model and questions on existence and uniqueness properties have not been 
considered there.

\subsection{Weighted $L^2$ Spaces.}\label{ss:weighted}
In connection with the Dirichlet type spaces
as presented in Subsection~\ref{ss:dirichlet}, but also from a more general 
perspective, it is natural to consider triplets associated to weighted $L^2$ spaces.
Let $(X;\fA)$ be a measurable space on which we consider a $\sigma$-finite
measure $\mu$.  A function 
$\omega$ defined on $X$ is called a \emph{weight} with respect to the 
measure space 
$(X;\fA;\mu)$ if it is measurable and $0<\omega(x)<\infty$, for $\mu$-almost all 
$x\in X$. Note that $\cW(X;\mu)$, the collection of weights with respect to 
$(X;\fA;\mu)$, is a multiplicative unital group. For an arbitrary
$\omega\in \cW(X;\mu)$, consider the measure $\nu$ whose Radon-Nikodym 
derivative with respect to $\mu$ is $\omega$, denoted $\de\nu=\omega\de\mu$,
that is, for any $E\in\fA$ we have
$\nu(E)=\int_E \omega\de\mu$.
It is easy to seee, e.g.\ see \cite{CojGh4}, that
 $\nu$ is always $\sigma$-finite. 

\begin{proposition}\label{p:weighted}
Let $\omega$ be a weight on the $\sigma$-finite measure space 
$(X;\fA;\mu)$ such that $\essinf_X\omega>0$. 
Let $\cH_0=L^2(X;\mu)$, $\cH_+=L^2_\omega(X;\mu)$ 
and $\cH_-=L^2_{\omega^{-1}}(X;\mu)$. 
Then $(\cH_+;\cH_0;\cH_-)$ 
is a triplet of Hilbert spaces for which:
\begin{itemize}
\item[(a)] The embeddings $j_\pm$ of $\cH_+$ in $\cH_0$ and
of $\cH_0$ in $\cH_-$ are bounded and have dense ranges.
\item[(b)] The adjoint $j_+^*$ is defined by $j_+^*h=\omega^{-1}h$ for all 
$h\in L^2(X;\mu)$.
\item[(c)] The kernel operator $A=j_+j_+^*$ is a nonnegative bounded 
operator defined by $Ah=\omega^{-1}h$, for all $h\in L^2(X;\mu)$. Moreover,
when viewed as an operator defined in $\cH_-$ and valued in $\cH_+$, $A$ 
admits a unique unitary extension $\widetilde A\colon \cH_-\ra\cH_+$.
\item[(d)] The Hamiltonian $H=A^{-1}$ is defined by $Hh=\omega h$ for all 
$h\in\dom(H)= L^2_{\omega^2}(X;\mu)$.
Moreover, when viewed as an operator defined in $\cH_+$ and valued in 
$\cH_-$, $H$ can be uniquely extended to a unitary operator 
$\widetilde H=\widetilde A^{-1}$.
\item[(e)] The canonical unitary identification of $\cH_+^*$ with 
$\cH_-$ is the operator $\Theta$ is defined by
\begin{equation}\label{e:ptheta} 
(\Theta g)(f):=\langle \widetilde A f,g\rangle_+= 
\int_X f\overline g\de\mu,\quad f\in\cH_+,\ g\in\cH_-, 
\end{equation}
\end{itemize}

Consequently, $\sigma(A)=\essran(\omega^{-1})$ and $\sigma(H)=
\essran(\omega)$, where $\essran$ denotes the $\mu$-essential range. 
\end{proposition}

A natural question that can be raised in connection with the preceding proposition
is whether anything might be said when dropping the assumption 
$\essinf\omega>0$. Again, the embeddings cannot be continuous anymore, and 
hence we have to allow unbounded operators to show up. Once
the notions of closed embeddings and triplets of closely embedded Hilbert
spaces have been singled out as in \cite{CojGh4}, Proposition~\ref{p:weighted}
can be naturally extended to cover the general case and we used
this extension in order 
to provide a solution to the construction of triplets of closely embedded Hilbert
spaces associated to any pair of Dirichlet type spaces, but questions on 
abstract models, existence and uniqueness properties, have not been 
considered yet.

\section{Notation and Preliminary Results}

A Hilbert space $\cH_+$
is called {\em closely embedded} in the Hilbert space $\cH$ if:
\begin{itemize}
\item[(ceh1)] There exists a linear manifold $\cD\subseteq
\cH_+\cap\cH$ that is dense in $\cH_+$.
\item[(ceh2)] The embedding operator $j_+$ with domain $\cD$ is closed,
as an operator $\cH_+\ra\cH$.
\end{itemize}

The meaning of
the axiom (ceh1) is that on $\cD$ the algebraic structures of $\cH_+$ and $\cH$
agree, while the meaning of the axiom (ceh2) is that the
embedding $j_+$ is explicitly defined by $j_+x=x$ for all 
$x\in\cD\subseteq \cH_+$ 
and, considered as an operator from $\cH_+$ to $\cH$, it is closed.
Also, recall that in case $\cH_+\subseteq\cH$ and the embedding operator 
$j_+\colon\cH_+\ra\cH$
is continuous, one says that $\cH_+$ is \emph{continuously embedded} in $\cH$,
e.g.\ see P.A.~Fillmore and J.P.~Williams~\cite{FlWl}
and the bibliography cited there.

Following L.~Schwartz \cite{Sch}, we call $A=j_+j_+^*$ the {\em
kernel operator} of the closely embedded Hilbert space $\cH_+$
with respect to $\cH$.

The abstract notion of closed embedding of Hilbert spaces was singled out
in \cite{CojGh3} following a generalized operator range model. In this section we
point out two models, which are dual in a certain way, and that will be 
used in this article as the main technical ingredient of the triplets 
of closely embedded Hilbert spaces. Constructions similar to those of the 
spaces $\cD(T)$ and $\cR(T)$ have been recently considered in the 
theory of interpolation of Banach spaces, e.g.\ see M.~Haase \cite{Haase} and
the rich bibliography cited there.

\subsection{The Space $\cD(T)$.}\label{ss:sdt} 
In this subsection we introduce a model of 
closely embedded Hilbert space generated by a closed densely defined operator. 
For the beginning, we consider a linear operator $T$ defined on a linear 
submanifold of $\cH$ and valued in $\cG$, for two Hilbert spaces $\cH$ and $\cG$,
and assume that its null space $\ker(T)$ is a closed subspace of $\cH$.
On the linear manifold $\dom(T)\ominus\ker(T)$ we consider 
the norm
\begin{equation}\label{e:normad} |x|_T:=\|Tx\|_\cG,\quad x\in\dom(T)
\ominus \ker(T),
\end{equation} and let $\cD(T)$ be the Hilbert space completion of the pre-Hilbert 
space $\dom(T)\ominus\ker(T)$ with respect to the norm $|\cdot|_T$  
associated the inner product $(\cdot,\cdot)_T$
\begin{equation}\label{e:ipd} (x,y)_T=\langle Tx,Ty\rangle_\cG,\quad 
x,y\in \dom(T)\ominus \ker(T).
\end{equation}
We consider the operator $i_T$ defined, 
as an operator in $\cD(T)$ and valued in $\cH$, as follows
\begin{equation}\label{e:iote} i_T x:=x,\quad x\in\dom(i_T)=\dom(T)\ominus\ker(T).
\end{equation}

\begin{lemma}\label{l:iclosed} The operator $i_T$ is closed if and only if $T$ is a
closed operator.
\end{lemma}

\begin{proof} Let us assume that $T$ is a closed operator. Then $\ker(T)$ is a 
closed subspace of $\cH$, hence the definition of the operator $i_T$ makes sense. 
In order to 
prove that $i_T$ is closed, let $(x_n)$ be a sequence in $\dom(i_T)$ such that 
$|x_n-x|_T\ra 0$ and $\|i_Tx_n-y\|_\cH\ra 0$, as $n\ra\infty$, for some $x\in\cD_T$ 
and $y\in \cH$. By \eqref{e:normad} it follows that the sequence 
$(Tx_n)$ is Cauchy in $\cG$. Since $(x_n)$ is also Cauchy in $\cH$, it follows
that the sequence of pairs $((x_n,Tx_n))$ is Cauchy in the graph norm of $T$ and 
then, since $T$ is a closed operator, it follows that there exists $z\in\dom(T)$ such
that
\begin{equation*} \|x_n-z\|_\cH+\|Tx_n-Tz\|_\cG\ra 0,\quad \mbox{ as }n\ra\infty.
\end{equation*} Taking into account that $\|Tx_n-Tz\|_\cG=|x_n-z|_T$ for all 
$n\geq 1$, we get $z=x$ modulo $\ker(T)$, hence $x\in\dom(i_T)$. In addition, 
$x=y$, hence $i_T$ is a closed operator.

The proof of the converse implication follows a similar reasoning as before.
\end{proof}

The next proposition emphasizes the fact that the construction of $\cD(T)$ is
actually a renorming process.

\begin{proposition}\label{p:isometry} The operator $T i_T$ admits a unique 
isometric extension $\widehat T\colon \cD(T)\ra \cG$.
\end{proposition}

\begin{proof} Since $\dom(i_T)=\dom(T)\ominus \ker(T)$ and $i_T$ acts like 
identity, it 
follows that $\dom(Ti_T)=\dom(i_T)$ which is dense in $\cD(T)$. Also,
for all $x\in \dom(i_T)$ we have $\|Ti_Tx\|_\cG=\|Tx\|_\cG=|x|_T$, hence $Ti_T$ is 
isometric. Therefore, $Ti_T$ extends uniquely to an isometric operator 
$\cD(T)\ra \cG$.
\end{proof}

The most interesting case is when the operator $T$ is a closed and densely 
defined operator in a Hilbert space $\cH$. 
The next proposition explores this case from the point of view of the
closed embedding of $\cD(T)$ in $\cH$ and that 
of the kernel operator $A=i_Ti_T^*$.

\begin{proposition}\label{p:dete} Let $T$ be a closed and densely defined 
operator on $\cH$ and valued in $\cG$, for two Hilbert spaces $\cH$ and $\cG$.

\nr{a} $\cD(T)$ is closely embedded in $\cH$ and $i_T$ is the underlying closed
embedding.

\nr{b} $\ran(T^*)\subseteq \dom(i_T^*)$ and equality holds 
provided that $\ker(T)=0$.

\nr{c}  $\ran(T^*T)\subseteq \dom(i_Ti_T^*)$ and equality holds provided that 
$\ker(T)=0$. In addition, 
\begin{equation}\label{e:ititstar} (i_Ti_T^*)(T^*T)x=x,\mbox{ for all }x\in \dom(T^*T)\ominus \ker(T)
\end{equation}

\nr{d} $(i_Ti_T^*)\ran(T^*T)\subseteq\dom(T^*T)$ and equality holds provided that 
$\ker(T)=0$. In addition, 
\begin{equation}\label{e:tetestart} (T^*T)(i_Ti_T^*)u=u,\mbox{ for all }u\in 
\ran(T^*T).\end{equation}
\end{proposition}

\begin{proof} (a) First note that, since $T$ is closed, its null space is closed, hence
the construction of the Hilbert space $\cD(T)$ and $i_T$ make sense. 
The operator $i_T$ is densely defined, by construction. By 
Lemma~\ref{l:iclosed}, $i_T$ is closed as well. Hence, the axioms (ceh1) and (ceh2)
are fulfilled.

(b) Let $y\in\ran(T^*)$ be arbitrary, hence $y=T^*x$ for some $x\in\dom(T^*)
\subseteq\cG$. Then, for all $u\in\dom(i_T)=\dom(T)\ominus\ker(T)$ we have
\begin{equation*} \langle y,i_Tu\rangle_\cH=\langle T^* x,y\rangle_\cH=\langle x,Tu
\rangle\cG,
\end{equation*} hence
\begin{equation*} |\langle y,i_Tu\rangle_\cH|\leq \|x\|_\cG\, \|Tu\|_\cG=\|x\|_\cG\, 
|u|_T,\quad.
\end{equation*} which implies that $y\in\dom(i_T^*)$.

Let us assume now that $\ker(T)=0$ and consider an arbitrary vector 
$y\in\dom(i_T^*)$. For any $x\in\ran(T)$ there exists a unique vector 
$u_x\in\dom(T)=\dom(i_T)$ such that $x=Tu_x$ and $\|x\|_\cG=|u_x|_T$. 
In this way, we can define a linear functional $\ran(T)\ni x\mapsto \phi_y(x)
=\langle i_Tu_x,y\rangle_\cH=\langle u_x,i_T^*y\rangle_T$ and note that
\begin{equation*} |u_x|_T\,|i_T^*y|_T=\|x\|_\cG \, |i_T^*y|_T,\quad x\in\ran(T).
\end{equation*} This shows that $\phi_y$ has a continuous extension 
$\widetilde\phi_y\colon \cG\ra\CC$ and hence, there exists $g\in\cG$ such that 
$\widetilde \phi_y(x)=\langle x,g\rangle_\cG$ for all $x\in\cG$. Specializing this
for arbitrary $x\in\ran(T)$, it follows that, on the one hand,
\begin{equation*} \widetilde\phi_y(x)=\langle x,g\rangle_G
=\langle Tu_x,g\rangle_\cG,
\end{equation*}
while, on the other hand,
\begin{equation*} \widetilde\phi_y(x)=\langle i_Tu_x,y\rangle_\cH
=\langle u_x,y\rangle_\cH.
\end{equation*} Since $\dom(T)$ is dense in $\cH$ it follows that $y=T^*g$, that is,
$y\in\ran(T^*)$.

(c) Let $y\in\ran(T^*T)$ be arbitrary, hence $y=T^*Tx$ for some $x\in\dom(T^*T)$,
that is, $x\in\dom(T)$ and $Tx\in\dom(T^*)$. Without loss of generality we can 
assume that $x\in\dom(T)\ominus\ker(T)=\dom(i_T)$. Then, for any 
$u\in\dom(i_T)$ we have
\begin{equation*} \langle y,i_Tu\rangle_\cH=\langle T^*Tx,y\rangle_\cH
=\langle Tx,Tu\rangle_\cG=(x,u)_T,
\end{equation*} hence, the linear functional 
$\cD(T)\supseteq \dom(i_T)\ni u\mapsto \langle 
i_Tu,y\rangle_\cH$ is bounded. Therefore, $y\in\dom(i_T^*)$ 
and $i_T^*y=x\in\dom(i_T)$, in particular, $y\in\dom(i_Ti_T^*)$. Thus, we showed
that $\ran(T^*T)\subseteq \dom(i_Ti_T^*)$ and that $(i_Ti_T^*)(T^*T)x=x$ for all 
$x\in\dom(T^*T)\ominus\ker(T)$ (recall that $\ker(T)=\ker(T^*T)$). 

If, in addition, $\ker(T)=0$, then $\ker(T^*T)=\ker(T)=0$ and then the 
representation $y=T^*Tx$ for $y\in\ran(T^*T)$ and $x\in\dom(T^*T)$ is unique and
the reasoning from above can be reversed, hence $\ran(T^*T)=\dom(i_Ti_T^*)$.

(d) As a consequence of the proof of (e),
we also get that $(i_Ti_T^*)$ maps $\ran(T^*T)$ in $\dom(T^*T)$ and that, for all
$u\in\ran(T^*T)$, we have $(T^*T)(i_Ti_T^*)u=u$. In case $\ker(T)=0$ then 
$(i_Ti_T^*)\ran(T^*T)=\dom(T^*T)$
\end{proof}

\begin{remark}\label{r:dete}
We can view the Hilbert space $\cD(T)$ and its closed embedding $i_T$ as a 
model for the abstract definition of a closed embedding. More precisely, let 
$(\cH_+;\|\cdot\|_+)$ be a Hilbert space closely embedded in the Hilbert space
($\cH;\|\cdot\|_\cH)$ and let $j_+$ denote the 
underlying closed embedding. Since $j_+$ is one-to-one, we can define a 
linear operator $T$ with $\dom(T)=\ran(j_+)\oplus \ker(j_+^*)$, viewed as a dense
linear manifold in $\cH$, and valued in $\cH_+$,
defined by $T(x\oplus x_0)=j_+^{-1}x$, for all 
$x\in\ran(j_+)$ and $x_0\in \ker(j_+^*)$. Then $\ker(T)=\ker(j_+^*)$ and, 
for all $x\in \ran(j_+)$ we have $x=j_+u$ for a unique $u=x\in\dom(j_+)$, hence
\begin{equation*} \|x\|_+=\|Tx\|_+=|x|_T.
\end{equation*}
Thus, modulo a completion of $\dom(j_+)$ which may be different, 
the Hilbert space $(\cD(T);|\cdot|_T)$ coincides with the Hilbert space 
$(\cH_+;\|\cdot\|_+)$.
\end{remark}

\subsection{The Hilbert Space $\cR(T)$.}\label{ss:thsrt}
In this subsection we recall
a construction and its basic properties
of Hilbert spaces associated to ranges of
general linear operators that was used in \cite{CojGh3} as 
the model that provided the abstract definition of a closed embedding of 
Hilbert spaces.

Let $T$ be a linear operator acting from a Hilbert space
$\cG$ to another Hilbert space $\cH$ and such that its null space $\ker(T)$ is 
closed. Introduce a pre-Hilbert space structure on $\ran(T)$ by the positive 
definite inner product $\langle \cdot,\cdot \rangle_T$ defined by
\begin{equation} \label{e:uvete}
\langle u,v \rangle_T = \langle x,y \rangle_\cG
\end{equation}
for all $u = T x$, $ v = T y$, $ x,y \in \dom (T)$ such that $x,y
\perp \ker (T)$. Let $\cR (T)$ be the completion of the pre-Hilbert
space $\ran (T)$ with respect to the corresponding norm $\| \cdot
\|_T$, where $\| u \|_T^{2}  = \langle u,u \rangle_T$, for $ u \in
\ran (T)$. The inner product and the norm on $\cR(T)$ are
denoted by $\langle \cdot,\cdot \rangle_T$ and, respectively, $\|\cdot\|_T$
throughout.

Further, consider the \emph{embedding operator} $j_T \colon\dom (j_T)
(\subseteq\cR (T)) \ra \cH$ with domain $\dom (j_T) = \ran (T)$ 
defined by
\begin{equation}\label{e:jete}
j_T u = u, \quad u \in \dom (j_T)=\ran(T).\end{equation}

Another way of viewing the definition of the Hilbert space $\cR(T)$
is by means of a certain factorization of $T$.

\begin{lemma}\label{l:cofactorization} Let $T$ be a linear operator
  with domain dense in the Hilbert space  $\cG$, valued in the Hilbert
  space $\cH$, and with closed null space. We consider the Hilbert
  space $\cR(T)$  and the embedding $j_T$ defined as in \eqref{e:uvete} and,
  respectively, \eqref{e:jete}. Then, there exists a unique coisometry
  $U_T\in\cB(\cG,\cR(T))$, such that
  $\ker(U_T)=\ker(T)$ and $T=j_TU_T$.
\end{lemma}

\begin{remark}\label{r:cofactorization}
The assumption in Lemma~\ref{l:cofactorization} that $T$ is densely defined is
not so important; if this is not the case then
$U_T$ must have a larger null space
only, in order to keep it unique. More precisely,
$\ker(U_T)=\ker(T)\oplus(\cG\ominus \dom(T))$ and,
consequently, $TP_{\overline{\dom(T)}}\subseteq j_TU_T$, 
which turns out to be an equality since $\ker(T)$ is supposed to be a closed 
subspace in $\cG$.
\end{remark}

The most interesting situation, from our point of view, is when the embedding
operator has some closability properties. 

\begin{lemma}\label{l:closed} Let $T$ be an operator densely defined in $\cG$, 
with range in $\cH$, and with closed null space. With the notation as before,
the operator $T$ is closed if and only if the embedding operator $j_T$ is
closed.\end{lemma}

We denote by $\cC(\cH,\cG)$ the collection of all operators $T$ that are closed 
and densely defined from $\cH$ and valued in $\cG$.
The following lemma is a direct consequence of Lemma~\ref{l:cofactorization} 
and Lemma~\ref{l:closed}.

\begin{lemma}\label{l:domtestar} Let $T\in\cC(\cH,\cG)$. Then $\dom(j_T^*)
\supseteq\dom(T^*)$. If, in addition, $T$ is one-to-one, then $\dom(j_T^*)
=\dom(T^*)$
\end{lemma}

We also recall an
extension of a characterization of operator ranges due to Yu.L.~Shmulyan \cite{Shm}
and similar results of L.~de~Branges and J.~Rovnyak \cite{deBraRov},
to the case of closed densely defined operators between Hilbert spaces, cf.\  
\cite{CojGh3}.

\begin{theorem}\label{t:range}
Let $T\in\cC(\cG,\cH)$ be nonzero and $u\in\cH$. Then
  $u\in \ran(T)$ if and only if there exists $\mu_u\geq 0$ such that
  $|\langle u,v\rangle_\cH| \leq \mu_u \|T^*v\|_\cG$ for all
  $v\in\dom(T^*)$. Moreover, if $u\in \ran(T)$ then
\begin{equation*} \|u\|_T=\sup\bigr\{ \frac{|\langle
  u,v\rangle_\cH|}{\|T^*v\|_\cG}  \mid
  v\in\dom(T^*),\ T^*v\neq 0\bigl\},\end{equation*}
where $\|\cdot\|_T$ is the norm associated to the inner product defined
as in \eqref{e:uvete}.
\end{theorem}

Let us observe that the definition of closely embedded Hilbert spaces is 
consistent with the model $\cR(T)$, for $T\in \cC(\cG,\cH)$, more precisely,  
if $\cH_+$ is closely embedded in $\cH$ then $\cR(j_+)=\cH_+$ and 
$\|x\|_+=\|x\|_{j_+}$.

The model for the abstract definition of closely embedded Hilbert spaces 
follows the results on the Hilbert space $\cR(T)$.
Thus, if
$T\in\cC(\cG,\cH)$ then the Hilbert space $\cR(T)$, with its canonical
embedding $j_T$ as defined in \eqref{e:uvete} and \eqref{e:jete}, is a Hilbert
space closely embedded in $\cH$, e.g.\ by
Lemma~\ref{l:closed}. Conversely, if $\cH_+$ is a Hilbert space closely
embedded in $\cH$, and $j_+$ denotes its canonical closed embedding, then
$\cH_+$ can be naturally viewed as the Hilbert space of type $\cR(j_+)$. This
fact is actually more general.

\begin{proposition}\label{p:tetestar}
Let $T\in\cC(\cG,\cH)$ and consider the Hilbert space
  $\cR(T)$ closely embedded in $\cH$,
with its canonical closed embedding $j_T$. Then $TT^*=j_Tj_T^*$.
\end{proposition}

As in the case of continuous embeddings, one can
prove that Hilbert spaces that are closely embedded in a given Hilbert
space are uniquely determined by their kernel operators, but the
uniqueness takes a slightly weaker form. This is illustrated by the following
theorem.

\begin{theorem}\label{t:kernel} Let $\cH_+$ be a Hilbert space
closely embedded in $\cH$, with $j_+ : \cH_+ \ra
\cH$ its densely defined and closed embedding operator, and let $A
= j_+ j_+^*$ be the kernel operator of $\cH_+$.
Then

\nr{a} $\ran(A^{1/2})=\dom(j_+)$ is dense in both $\cR(A^{1/2})$ and $\cH_+$.

\nr{b} For all $x\in\ran(A^{1/2})$ and all $y\in \dom(A)$ we have $\langle
x,y\rangle_\cH=\langle x, Ay\rangle_+=\langle
x,Ay\rangle_{A^{1/2}}$.

\nr{c} $\ran(A)$ is dense in both $\cR (A^{1/2})$ and $\cH_+$.

\nr{d} For any $x\in\dom(j_+)$ we have
\begin{equation*} \|x\|_+=\sup\bigl\{ \frac{|\langle
    x,y\rangle_\cH|}{\|A^{1/2}y\|_\cH} \mid y\in\dom(A^{1/2}),\ A^{1/2}y\neq
    0\bigr\}.
\end{equation*}

\nr{e} The identity operator $:\ran(A))(\subseteq\cR (A^{1/2})) \ra
\cH_+$ uniquely extends to a unitary operator $V:\cR(A^{1/2})\ra \cH_+$
such that $
V A x = j_+^* x$, for all $ x \in \dom(A)$.
 \end{theorem}

\section{A Model of a Triplet of Closely Embedded Hilbert Spaces} 
\label{s:model}

In this section we develop a 
construction of a chain of two closed embeddings with certain duality properties 
related to a given positive selfadjoint operator with trivial null space, as a 
generalization of the classical notion of a triplet of Hilbert spaces. This 
construction will lead us to the axiomatization of triplets of
closely embedded Hilbert spaces and will be essential in applications. 
Let $\cH$ be a Hilbert space and $H$ a positive selfadjoint operator in 
$\cH$, that we call the Hamiltonian. We assume that $H$ has trivial null space. 
Let $\cG$ be another Hilbert space and let 
$T\in\cC(\cH,\cG)$ be such that it provides a factorization of the Hamiltonian
\begin{equation}\label{e:hamfact} H=T^*T.
\end{equation} Then $T$ has trivial null space as well, and let $T^{-1}$ denote the 
algebraic inverse operator of $T$, that is, $\dom(T^{-1})=\ran(T)$.
We consider the Hilbert space $\cD(T)$ as described in 
Subsection~\ref{ss:sdt}, more precisely, in our special case $\cD(T)$ is the Hilbert 
space completion of $\dom(T)$ with respect to the quadratic norm 
$|\cdot|_T$ defined 
as in \eqref{e:normad}, and the associated inner product $(\cdot,\cdot)_T$. 
The closed embedding $i_T$, defined as in \eqref{e:iote}, has 
domain $\dom(T)$ dense in $\cD(T)$ and range in $\cH$. Observe that, without 
loss of generality, we can assume that $T$ has dense range (otherwise, replace 
$\cG$ by 
the closure of $\ran(T)$). For example, all these assumptions are met when 
$T=H^{1/2}$, and uniqueness modulo unitary equivalence holds as well, 
but having in mind future applications we want to keep this level of generality. 

Throughout this section we keep the following two assumptions on 
$T$: \emph{$\ker(T)=\{0\}$ and $\ran(T)$ is dense in $\cG$}.
As mentioned in Subsection~\ref{ss:sdt}, 
the kernel operator $A$ of the closed embedding $i_T$ is a positive selfadjoint
operator in $\cH$ 
\begin{equation}\label{e:aiet}A=i_Ti_T^*=j_{T^{-1}} j_{T^{-1}}^*=T^{-1}
{T^{-1}}^*=(T^*T)^{-1}
\end{equation} hence, in accordance with \eqref{e:hamfact}, $H=T^*T=A^{-1}$; the 
kernel operator is the inverse of the Hamiltonian, in the sense of one-to-one 
unbounded operators. 

In the following we use Lemma~\ref{l:cofactorization}. Thus, we have the 
coisometry $V_T\in\cB(\cG,\cD(T))$, uniquely determined such that 
$T^{-1}=i_TV_T$ and $\ker(V_T)=\cG\ominus\ran(T)$. 
Due to our assumption that $\ran(T)$ is dense in $\cG$, 
the operator $V_T$ is actually unitary. Similarly, there exists a coisometry
$U_{T^*}\in\cB(\cG,\cR(T^*)$ such that $T^*=j_{T^*}U_{T^*}$, uniquely determined 
by the property $\ker(U_{T^*})=\ker(T^*)$. Again, since $\ran(T)$ is supposed 
to be dense in $\cG$, it follows that $U_{T^*}$ is actually unitary.

The kernel 
operator $B$ of the closed embedding of $\cH$ in $\cR(T^*)$ is
\begin{equation}\label{e:kerbe} B=j_{T^*}^{-1}
{j_{T^*}^{-1}}^*=(j_{T^*}^*j_{T^*})^{-1}.
\end{equation} On the other hand, since 
$T^*=j_{T^*}U_{T^*}$, where $U_{T^*}\colon \cG\ra\cR(T^*)$ is unitary, it follows
that
\begin{equation*} TT^*=U_{T^*}^* j_{T^*}^* j_{T^*}U_{T^*},
\end{equation*} which, when combined with \eqref{e:kerbe}, shows that 
\begin{equation}\label{e:beiet}(TT^*)^{-1}=
U_{T^*}^* (j_{T^*}^*j_{T^*})^{-1}U_{T^*}=U_{T^*}^*BU_{T^*}.
\end{equation} Since, via the polar decomposition for the closed densely defined 
operator $T$, the operators $TT^*$ and $T^*T$ are unitary equivalent, from 
\eqref{e:aiet} and \eqref{e:beiet} it follows that 
the two kernel operators $A$ and $B$ are unitary equivalent.

Further on, consider the unitary operator $U_{T^*}V_T^{-1}$, acting between 
$\cD(T)$ and $\cR(T^*)$, and denote this operator by $\widetilde H$. Then, 
$\widetilde H$ is an extension of the 
Hamiltonian operator $H$ and its inverse, that we denote by $\widetilde A$, 
is an extension of the 
kernel operator $A$. Indeed, this follows from the fact that 
$T^*T=j_{T^*}U_{T^*}V_{T}^{-1}i_T^{-1}$, and then taking into account of 
\eqref{e:hamfact}, and the fact that both $j_{T^*}$ and $i_T$ are 
closed embeddings.

Let us observe now that the kernel operator can be viewed as an operator acting
from $\cR(T^*)$ and valued in $\cD(T)$. Indeed, taking into account \eqref{e:aiet},
$\dom(A)=\dom(i_Ti_T^*)\subseteq \ran(T^*)\subseteq \cR(T^*)$ and 
$\ran(A)\subseteq \dom(T)\subseteq\cD(T)$. Since $H=A^{-1}$, it follows that
the Hamiltonian operator $H$ can be viewed as acting from $\cD(T)$ and valued
in $\cR(T^*)$. 

In the following we show that the operator $H$, 
when viewed as an operator acting
from $\cD(T)$ and valued in $\cR(T^*)$, is densely defined and has dense range.
Indeed, in order to prove that the domain of $H$ is dense in $\cD(T)$ it is sufficient 
(actually, equivalent) to proving that $\ran(A)$ is dense in $\cD(T)$. To see this, let 
$x\in \cD(T)$ be such that $(x,Ay)_T=0$ for all $y\in\dom(A)$. We first prove that 
$(x,i_T^*)_T=0$ for all $y\in\dom(i_T^*)$. Indeed, since $A=i_Ti_T^*$, it follows 
that $\dom(A)$ is a core for $i_T^*$, hence, for any $y\in\dom(i_T^*)$ there exists
a sequence $(y_n)$ of vectors in $\dom(A)$ such that $\|y_n-y\|_\cH\ra 0$ and 
$|i_T^*y-i_T^* y_n|_T\ra 0$ as $n\ra\infty$. Consequently, 
$0=(x,Ay_n)_T=(x,i_T^*y_n)_T\ra (x,i_T^*y)_T$ as $n\ra\infty$, hence 
$(x,i_T^*y)_T=0$. Since $y$ is arbitrary in $\dom(i_T^*)$ and $\ran(i_T^*)$ is 
dense in $\cD_T$, it follows that $x=0$. Thus, $\ran(A)=\dom(H)$ is dense in 
$\cD(T)$. In a completely similar fashion, by using $j_{T^*}$ instead of $i_T$ and 
taking into account that $H=T^*T$, we prove that $\ran(H)$ is dense in $\cR(T^*)$.

The construction we got so far 
can be visualized by the compound diagram in Figure 1, where 
all the triangular diagrams are commutative, by definition, while the rectangular 
diagram is commutative in the weaker sense $j_{T^*}\widetilde H\supseteq H i_T$.
\medskip

\hfil\xymatrix{
\ \cG\ \ar[d]_{V_T} \ar@{<-}[rd]^{T} & &\ \cG\ \ar[ld]_{T^*} 
\ar[d]^{U_{T^*}} & \\
\ \cD(T)\ \ar[r]^{i_T} &\ \cH\ \ar@<-0.5ex>[d]^{\ H=A^{-1}}  \ar[dl]_{i_T^*} 
\ar@<0.5ex>[r]_{j_{T^*}}&\ \cR(T^*)\ \ar@{-->}[l]_{j_{T^*}^{-1}} \ar@<-0.5ex>[d]_{\widetilde A} & \\
\ \cD(T)\ \ar[r]_{i_T} &\ \cH\ \ar@<0.5ex>[r]_{i_{T}} \ar@<-0.5ex>[u]^{A\ } &\ \cD(T)\ 
\ar@<-0.5ex>[u]_{\widetilde H} \ar@{-->}[l]_{i_{T}^{-1}} &\ \cG\ \ar[lu]_{U_{T^*}} 
\ar[l]_{V_T}
}\hfill

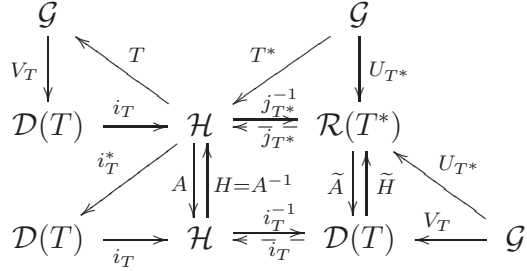
\smallskip\centerline{Figure 1.}\medskip

Let us observe now that, as a consequence of 
Theorem~\ref{t:range} when applied to $T^*$ instead of $T$,
for all $y\in\dom(T^*)$ we have the following variational formula
\begin{equation}\label{e:varnorm}
\|y\|_{T^*}=\sup\bigl\{
\frac{|\langle y,x\rangle_\cH|}{|x|_T}\mid x\in\dom(T)
\setminus\{0\}\bigr\}.
\end{equation}

Finally, we show that there is a canonical identification of $\cR(T^*)$ with the  
conjugate dual space $\cD(T)^*$. To see this, we define a linear operator 
\begin{equation}\label{e:theta}
\Theta \colon \cR(T^*)\ra\cD(T)^*,\quad (\Theta\alpha)(x):=(\widetilde A\alpha,x)_T,
\quad \alpha\in\cR(T^*),\ x\in\cD(T),
\end{equation} and, taking into account that $\widetilde A$ is unitary it follows 
that $\Theta$ is unitary as well.

We summarize all the previous constructions and facts in the following

\begin{theorem}\label{t:model} Let $H$ be a positive selfadjoint operator in the Hilbert 
space $\cH$, with trivial null space. Let $T\in\cC(\cH,\cG)$ be such that 
$\ran(T)$ is dense in $\cG$ and $H=T^*T$. Then:
\begin{itemize}
\item[(i)] The Hilbert space $\cD(T)$ is closely embedded in $\cH$ with its closed 
embedding $i_T$ having range dense in $\cH$, and its kernel operator 
$A=i_Ti_T^*$ 
coincides with $H^{-1}$.
\item[(ii)] $\cH$ is closely embedded in the Hilbert space $\cR(T^*)$ with its closed 
embedding $j_{T^*}^{-1}$ having range dense in $\cR(T^*)$. The kernel operator 
$B=j_{T^*}^{-1}{j_{T^*}^{-1*}}$ of this closed embedding is unitary equivalent 
with $A=H^{-1}$.
\item[(iii)] The operator $i_T^*|\ran(T^*)$ 
extends uniquely to a unitary operator 
$\widetilde A$ between the Hilbert spaces $\cR(T^*)$ and $\cD(T)$. In addition,
$\widetilde A$ is the unique unitary extension of the kernel operator $A$, when 
viewed as an operator acting from $\cR(T^*)$ and valued in $\cD(T)$, as well.
\item[(iv)] The operator $H$ can be viewed as a linear operator with domain dense 
in $\cD(T)$ and dense range in $\cR(T^*)$, is isometric, extends uniquely 
to a unitary operator 
$\widetilde H\colon \cD(T)\ra\cR(T^*)$, and $\widetilde H={\widetilde A}^{-1}$.
\item[(v)] Letting $V_T\in\cB(\cG,\cD_T)$ denote the unitary operator such that 
$T^{-1}=i_T V_T$ and $U_{T^*}\in\cB(\cG,\cR(T^*))$ denote the unitary 
operator such 
that $T^*=U_{T^*}j_{T^*}$, we  have $\widetilde H=U_{T^*}V_T^{-1}$.
\item[(vi)] The operator $\Theta$ defined by \eqref{e:theta} provides a canonical 
identification of the Hilbert space $\cR(T^*)$ with the conjugate dual
space $\cD(T)^*$ and, for all $y\in\dom(T^*)$
\begin{equation*}
\|y\|_{T^*}=\sup\bigl\{
\frac{|\langle y,x\rangle_\cH|}{|x|_T}\mid x\in\dom(T)
\setminus\{0\}\bigr\}.
\end{equation*}
\end{itemize}
\end{theorem}

\section{Triplets of Closely Embedded Hilbert Spaces}\label{s:tcehs}

In this section, we use the model obtained in Theorem~\ref{t:model} in order to 
derive an abstract definition for a triplet of closely embedded Hilbert spaces and 
then we approach existence, uniqueness, and other basic properties, 
as a left-right symmetry.

\subsection{Definition and Basic Properties}\label{ss:dbp}
By definition, $(\cH_+;\cH_0;\cH_-)$ is called a \emph{triplet of closely embedded 
Hilbert spaces} if:

\begin{itemize}
\item[(th1)] $\cH_+$ is a Hilbert space 
closely embedded in the Hilbert space
$\cH_0$, with the closed embedding denoted by $j_+$, and such that $\ran(j_+)$ is 
dense in $\cH_0$.
\item[(th2)] $\cH_0$ is closely embedded in the Hilbert space $\cH_-$, 
with the closed embedding denoted by $j_-$, and such that $\ran(j_-)$ is dense in 
$\cH_-$.
\item[(th3)] $\dom(j_+^*)\subseteq\dom(j_-)$ and for every vector
$y\in\dom(j_-)\subseteq \cH_0$ we have
\begin{equation}\label{e:dual}
 \|y\|_-=\sup\bigl\{\frac{|\langle x,y\rangle_{\cH_0}|}{\|x\|_+}\mid 
x\in\dom(j_+),\ x\neq 0\bigr\}.
\end{equation}
\end{itemize}

Let us first observe that, by \eqref{e:dual} in axiom (th3), 
for all $y\in\dom(j_-)$ and $x\in\dom(j_+)$ we have 
$|\langle j_+x,y\rangle_{\cH_0}|=|\langle x,y\rangle_{\cH_0}|\leq \|x\|_+ \|y\|_-$. By the definition
of $\dom(j_+^*)$ this means that $\dom(j_-)\subseteq \dom(j_+^*)$ hence, 
taking into account of $\dom(j_+^*)\subseteq \dom(j_-)$, the first condition 
in axiom (th3), it follows that actually 
\begin{equation}\label{e:domeq}\dom(j_+^*)=\dom(j_-).\end{equation}

In the following we show that the axioms (th1)--(th3) are sufficient in order to obtain
essentially all the properties that we get in Theorem~\ref{t:model}. 
Given $(\cH_+;{\cH_0};\cH_-)$ a triplet of closely embedded Hilbert spaces and letting 
$j_\pm$ denote the closed embedding of $\cH_+$ in ${\cH_0}$ and, respectively, 
the closed embedding of ${\cH_0}$ in $\cH_-$, the operator $A=j_+j_+^*$ is positive
selfadjoint in ${\cH_0}$ and it is called the \emph{kernel operator}. Also, 
since $\ran(j_+)$ 
is dense in ${\cH_0}$, it follows that $\ran(A)$ is dense in ${\cH_0}$ as well, equivalently 
$\ker(A)=\{0\}$. In particular, $H:=A^{-1}$ is a positive
selfadjoint operator in ${\cH_0}$ 
and it is called the \emph{Hamiltonian} of the triplet $(\cH_+;{\cH_0};\cH_-)$. 
Clearly, $0$ is not an eigenvalue of $H$. In addition, 
let us observe that $\dom(H)\subseteq \ran(j_+)=\dom(j_+)\subseteq \cH_+$

Further on, for any $y\in\ran(j_-)$, the linear 
functional $\cH_+\supseteq\ran(j_+)\ni x\mapsto \langle x,y\rangle_{\cH_0}\in\CC$ is 
bounded and hence, via the Riesz Representation Theorem, there exists uniquely 
$z_y\in\cH_+$ such that $\langle x,y\rangle_{\cH_0}=\langle x,z_y\rangle_{\cH_+}$ for all 
$x\in\ran(j_+)=\dom(j_+)$, and $\|z_y\|_{+}=\|y\|_{-}$. Thus, a linear operator 
$V\colon \dom(j_-)(\subseteq\cH_-)\ra \cH_+$ is uniquely defined by $Vy=z_y$, and it 
is isometric, in particular it is extended uniquely to an isometry 
$\widetilde V\colon \cH_-\ra\cH_+$. In addition, for all $x\in\dom(j_+)=\ran(j_+)$ and all 
$y\in\dom(j_-)=\ran(j_-)$ we have
\begin{equation*} \langle j_+x,y\rangle_{\cH_0}=\langle x,y\rangle_{\cH_0}
=\langle x,z_y\rangle_+=\langle x,Vy\rangle_+,
\end{equation*} that is, $V$ is $j_+^*$ when viewed as a linear operator from 
$\cH_-$ and valued in $\cH_+$. Consequently, 
$\ran(V)\supseteq \ran(j_+^*)$, which is dense in $\cH_+$. Thus, we have shown that
the isometric operator $\widetilde V$ is actually unitary $\cH_-\ra\cH_+$.

We observe that the kernel operator $A$ can be viewed also as acting from 
$\cH_-$ and valued in $\cH_+$. Indeed, $A=j_+j_+^*$,  hence 
$\dom(A)\subseteq \dom(j_+^*)=\dom(j_-)\subseteq\cH_-$ and, clearly, 
$\ran(A)\subseteq \ran(j_+)\subseteq \cH_+$. On the other hand, for any 
$y\in\dom(A)\subseteq \cH_-$ and any $x\in\dom(j_+)\subseteq\cH_+$ we have 
$\langle Ay,x\rangle_+=\langle j_+j_+^*y,x\rangle_+=\langle j_+^*y,x\rangle_+$, hence 
$A$ is a restriction of the operator $V$ defined before. 

In the following we prove that $\ran(A)$ is dense in $\cH_+$. To see this, 
let $x\in\cH_+$ be such that $\langle x,Ay\rangle_+=0$ for all $y\in\dom(A)$. We claim 
that $\langle x,j_+^*y\rangle_+=0$ for all $y\in\dom(j_+^*)$. 
Indeed, since $\dom(j_+^*)$ 
is a core for $A$, it follows that for any $y\in\dom(j_+^*)$ there exists a sequence 
$(y_n)$ of vectors in $\dom(A)$ such that $\|y_n-y\|_{\cH_0}\ra 0$ and 
$\|j_+^*y_n-j_+^*y\|_+\ra 0$ as $n\ra\infty$, hence $0=\langle x,Ay\rangle_+
=\langle x,j_+^*y_n\rangle_+\ra \langle x,j_+^*y\rangle_+$ as $n\ra\infty$. 
Taking into account
that the range of $V=j_+^*$, considered as an operator from $\cH_-$ to $\cH_+$, 
is dense in $\cH_+$, it follows that $x=0$. Thus, we conclude that $\ran(A)$ 
is dense in $\cH_+$.

In a similar fashion we can prove that $\dom(A)$ is dense in $\cH_-$. Since $A$, when 
viewed as a linear operator from $\cH_-$ to $\cH_+$, is a restriction of the operator 
$V$ (formally the same with $j_+^*$) which is isometric, it follows that the linear 
operator $A$,  when viewed as a linear operator from $\cH_-$ to $\cH_+$, is isometric 
and that it has a unique unitary extension $\widetilde A\colon \cH_-\ra\cH_+$, which is
exactly $\widetilde V$.

Similarly, the Hamiltonian operator can be viewed as a 
linear operator densely defined in $\cH_+$ and with range in $\cH_-$: 
recall that $\dom(j_+^*)=\dom(j_-)$ and hence that it is a subspace of $\cH_-$. Since
$H=A^{-1}$, it follows that $H$ is a restriction of $V^{-1}$, it is isometric, with domain 
dense in $\cH_+$ and range dense in $\cH_-$, hence it has a unique unitary extension 
$\widetilde H={\widetilde A}^{-1}={\widetilde V}^{-1}\colon \cH_+\ra\cH_-$.

For a better understanding of all these proven facts we depict these constructions 
by the following diagram:\medskip

\hfil\xymatrix{   &\dom(j_+^*) \ar[ld]_{V} \ar@{^{(}->}[d] \ar@{_{(}->}[rd] & \\
\cH_+ \ar[r]^{j_+} & \ {\cH_0}\ \ar@<-0.5ex>[d]^{\ H=A^{-1}}  \ar[dl]_{j_+^{-1}} \ar@<0.5ex>[r]_{j_-^{-1}}& \ \cH_-\ar@{-->}[l]_{j_-} \ar@<-0.5ex>[d]_{\widetilde V=\widetilde A} \\
\ \cH_+ \ar[r]_{j_+} &\ {\cH_0}\ \ar@<0.5ex>[r]_{j_+} \ar@<-0.5ex>[u]^{A\ } & \ \cH_+ \ar@<-0.5ex>[u]_{\widetilde H={\widetilde A}^{-1}} \ar@{-->}[l]_{j_+^{-1}} 
}\hfill

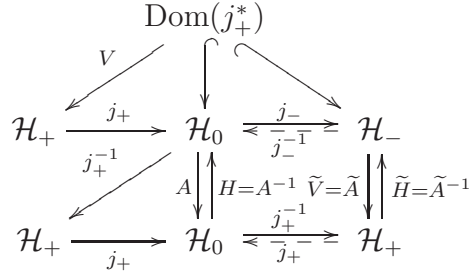
\smallskip\centerline{Figure 2.}\medskip

\noindent In Figure 2, all the triangular diagrams are commutative, by 
definition. The lower right rectangular diagram is commutative in a weaker sense, 
namely $j_- H\subseteq\widetilde H j_+^{-1}$.

Finally, we show that there exists a natural identification of $\cH_-$ with the 
conjugate dual space of $\cH_+$, more precisely, we consider the operator 
$\Theta\colon \cH_-\ra\cH_+^*$ defined by
\begin{equation*} (\Theta y)(x):=\langle \widetilde Vy,x\rangle_+,\quad 
y\in\cH_-,\ x\in \cH_+.
\end{equation*}
To see this, note that for any $l\in\cH_+^*$ there exists uniquely $z\in\cH_+$ such that
$ l(x)=\langle z,x\rangle_+$, for all $x\in\cH_+$. Letting $y={\widetilde V}^{-1}z
\in\cH_-$ it follows
\begin{equation*} l(x)=\langle z,x\rangle_+=\langle \widetilde Vy,x\rangle_+=
(\Theta y)(x),\quad x\in\cH_+.
\end{equation*} Thus, $\Theta$ is surjective. In addition, with the notation as before,
we have
\begin{equation*}\|\Theta y\|=\|\widetilde Vy\|_+=\|y\|_-,\quad y\in\cH_-,
\end{equation*} hence $\Theta$ is unitary, as claimed.

We gather all these proven facts in the following

\begin{theorem}\label{t:triplet} Let $(\cH_+;{\cH_0};\cH_-)$ be a triplet of closely embedded 
Hilbert spaces, and let $j_\pm$ denote the corresponding closed embeddings 
of $\cH_+$ in ${\cH_0}$ and, respectively, of ${\cH_0}$ in $\cH_-$. Then:

\nr{a} The kernel operator $A=j_+j_+^*$ is positive selfadjoint in ${\cH_0}$ and $0$ 
is not an eigenvalue for $A$. Also, the Hamiltonian operator $H=A^{-1}$ is a 
positive selfadjoint operator in ${\cH_0}$ for which $0$ is not an eigenvalue.

\nr{b} $\dom(j_+^*)=\dom(j_-)$, the closed embeddings $j_+$ and $j_-$ are 
simultaneously continuous or not, and the operator $V=j_+^*\colon\dom(j_+^*)
(\subseteq\cH_-)\ra\cH_+$ extends uniquely to a unitary operator 
$\widetilde V\colon \cH_-\ra\cH_+$. 

\nr{c} The kernel operator $A$ can be viewed as an operator densely defined in 
$\cH_-$ with dense range in $\cH_+$, and it is a restriction of the unitary operator 
$\widetilde V$.

\nr{d} The Hamiltonian operator $H$ can be 
viewed as an operator densely defined in $\cH_+$ with range dense in 
$\cH_-$, and it is uniquely extended to a unitary operator 
$\widetilde H\colon \cH_+\ra\cH_-$, and $\widetilde H=\widetilde{V}^{-1}$.

\nr{e} 
The operator $\Theta$ defined by $(\Theta y)(x)=\langle\widetilde Vy,x\rangle_+$, 
for all $y\in\cH_-$ and all $x\in\cH_+$ provides a unitary identification of $\cH_-$ 
with the conjugate dual space $\cH_+^*$.
\end{theorem}

\subsection{Existence and Uniqueness}\label{ss:eu}
We can now approach questions related to existence and 
uniqueness of triplets of closely embedded Hilbert spaces, similar to results
known for the classical triplets of Hilbert spaces, cf.\ \cite{Berezanski}. 
First we show that, 
in a triplet of closely embedded Hilbert spaces $(\cH_+;{\cH_0};\cH_-)$, 
the essential part, in a weaker sense, 
is the left-hand one, that is, the closed embedding of $\cH_+$ into ${\cH_0}$.

\begin{theorem}\label{t:existence2} Assume that ${\cH_0}$ and $\cH_+$ are two 
Hilbert spaces such that $\cH_+$ is closely embedded in ${\cH_0}$, with $j_+$ denoting 
this closed embedding, and such that $\ran(j_+)$ is dense in ${\cH_0}$.
 
\nr{1} One can always extend this closed embedding to the triplet 
$(\cH_+;{\cH_0};\cR(j_+^{-1*}))$ of closely embedded Hilbert spaces.

\nr{2} Let $(\cH_+;{\cH_0};\cH_-)$ be any other extension of the closed embedding 
$j_+$ to a triplet of closely embedded Hilbert spaces, let $A=j_+j_+^*$ be its 
kernel operator, and let $j_-$ denote the closed embedding of ${\cH_0}$ in $\cH_-$. 
Then, there exists a unique unitary operator 
$\Phi_-\colon \cH_-\ra\cR(j_+^*)$ such that when restricted to $\dom(j_-)$ acts as 
the identity operator. 
\end{theorem}

\begin{proof} (1) Indeed, the kernel operator $A=j_+j_+^*$ of $\cH_+$
is a positive selfadjoint operator
in ${\cH_0}$ and it is one-to-one, since $\ran(j_+)$ is supposed to be dense in ${\cH_0}$. 
Then $H=A^{-1}$ is a one-to-one positive selfadjoint operator in ${\cH_0}$ and 
letting $T=j_+^{-1}$ we have $H=T^*T$, with $T$ closed, densely defined, and
one-to-one, as an operator from ${\cH_0}$ into $\cH_+$. Then 
we apply Theorem~\ref{t:model}, more precisely, we define 
$\cH_-=\cR(T^*)=\cR({j_+^{-1}}^*)$.

(2) Since $\dom(j_+^*)=\dom(j_-)$ we can use the operators $\widetilde V$ in 
Theorem~\ref{t:model} and Theorem~\ref{t:triplet} to prove that the identity operator on 
$\dom(j_-)$ when viewed as a linear operator from $\cH_-$ and with range in 
$\cR(j_+^*)$ extends uniquely to a unitary operator.
\end{proof}

As a consequence of the previous theorem we can prove that 
the concept of triplet of Hilbert spaces with closed embeddings has a certain 
"left-right" symmetry, which, in general, the classical triplets of Hilbert spaces do 
not share.

\begin{proposition}\label{p:symmetry}
Let $(\cH_+;\cH_0;\cH_-)$ be a triplet of closely embedded Hilbert spaces. Then 
$(\cH_-;\cH_0;\cH_+)$ is also a triplet of closely embedded Hilbert spaces, more 
precisely:

\nr{1} If $j_+$ and $j_-$ denote the closed embeddings of $\cH_+$ in $\cH_0$ and, 
respectively, of $\cH_0$ in $\cH_-$, then $j_-^{-1}$ and $j_+^{-1}$ are the closed 
embeddings of $\cH_-$ in $\cH_0$ and, respectively, of $\cH_0$ in $\cH_+$.

\nr{2} If $H$ and $A$ denote the Hamiltonian, respectively, the kernel operator of the 
triplet $(\cH_+;\cH_0;\cH_-)$, then $A$ and $H$ are the Hamiltonian and, 
respectively, the kernel operator of the triplet $(\cH_-;\cH_0;\cH_+)$.
\end{proposition}

\begin{proof} We first prove the statement assuming that the given triplet is in the 
model form, that is, for some Hilbert space $\cG$ and some 
operator $T\in\cC(\cH_0,\cG)$ that is one-to-one and has dense range, we have 
$\cH_+=\cD(T)$ and $\cH_-=\cR(T^*)$, with the closed embeddings $j_+=i_T$ and, 
respectively, $j_-=j_{T^*}^{-1}$, as in Section~\ref{s:model}. Then, observe that 
$S={T^*}^{-1}\in\cC(\cH_0,\cG)$ is one-to-one and has 
dense range and that, inspecting the corresponding constructions in subsections 
\ref{ss:sdt} and \ref{ss:thsrt}, we have $\cD(S)=\cR(T^*)=\cH_-$ and 
$\cR(S^*)=\cD(T)=\cH_+$. By Theorem~\ref{t:model} it follows that 
$(\cH_-;\cH_0;\cH_+)$ is now a triplet of closely embedded Hilbert spaces as well, 
with closed embeddings $j_{T^*}$ and, respectively, $i_T^{-1}$. Thus, assertion (1)
is proven, in this special case. 

In order to prove assertion (2), note that, by Proposition~\ref{p:tetestar}, we have 
$j_{T^*}j_{T^*}^*=T^*T=H$, hence $H$ is the kernel operator of the triplet 
$(\cH_-;\cH_0;\cH_+)$, and then $A$ becomes its Hamiltonian operator.

The general case now follows from assertion (2) in Theorem~\ref{t:existence2} that
shows that, without loss of generality, we 
can assume that $\cH_+=\cD(T)$ and $\cH_-=\cR(T^*)$ for some 
$T\in\cC(\cH_0,\cG)$ which is one-to-one and has dense range, more precisely, we 
can take $\cG=\cH_+$ and $T=j_+$, the closed embedding of $\cH_+$ in $\cH_0$.
\end{proof}

We are now in a position to approach existence and uniqueness of triplets of 
Hilbert spaces in terms of a given Hamiltonian operator.

\begin{theorem}\label{t:existence1} Let $H$ be an arbitrary 
positive selfadjoint operator in a Hilbert space ${\cH_0}$ for  which $0$ is not an 
eigenvalue.

\nr{1} With notation as in subsections~\ref{ss:sdt} 
and \ref{ss:thsrt}, $(\cD(H^{1/2});{\cH_0};\cR(H^{1/2}))$ is a triplet of closely embedded 
Hilbert spaces such that $H$ is its Hamiltonian.

\nr{2} Let $(\cH_+;{\cH_0};\cH_-)$ be any other triplet of closely embedded Hilbert spaces
with the same Hamiltonian $H$. Then:

\begin{itemize}
\item[(a)] $\dom(H^{1/2})$ is dense in both $\cD(H^{1/2})$ and $\cH_+$.
\item[(b)] For any $x\in\dom(H^{1/2})$ and any $y\in\dom(H)$ we have $\langle x,Hy\rangle_{\cH_0}=\langle x,y\rangle_+=(x,y)_{H^{1/2}}$.
\item[(c)] $\dom(H)$ is dense in both $\cD(H^{1/2})$ and $\cH_+$.
\item[(d)] For any $x\in\dom(j_+)=\dom(H^{1/2})$ we have
\begin{equation*} \|x\|_+=\sup\left\{ \frac{|\langle x,H^{1/2}z\rangle_{\cH_0}|}{\|z\|_{\cH_0}} \mid z\in\dom(H^{1/2})\right\}.\end{equation*}
\item[(e)] The identity operator $:\dom(H)(\subseteq\cD(H^{1/2})\ra\cH_+$ extends uniquely to a unitary operator $\Phi_+\colon \cD(H^{1/2})\ra \cH_+$ such that $\Phi_+ y=j_+^*Hy$ for all $y\in\dom(H)$.
\end{itemize}
\end{theorem}

\begin{proof} (1) Indeed, we can apply Theorem~\ref{t:model} to $T=H^{1/2}$, 
since $T$ is one-to-one as well.

(2) The argument is essentially contained in Theorem~\ref{t:kernel}, only that this 
is rephrased in terms of the Hamiltonian $H$ instead of its inverse, the kernel 
operator $A$.
\end{proof}

\section{Weak Solutions for a Class of Dirichlet Problems}
\label{s:wsg}

In this section we apply the abstract results on triplets of closely embedded Hilbert
spaces to weak solutions for a Dirichlet problem associated to a class of degenerate 
elliptic partial differential equations.
We briefly fix the notation and recall some of the underlying facts related to Sobolev
spaces.  
Let $\Omega$ be an open   (nonempty) set of the $N$-dimensional euclidean  
space $\mathbb{R}^{N}$.  We use the notation  $D_{j} = \iac \frac{\partial}{\partial 
x_{j}}$, $(j =1, \ldots, N)$ for the operators of  differentiation with respect to the 
coordinates  of points 
$x = (x_{1}, \dots, x_{N})$ in $\mathbb{R}^{N}$, and, for a  multi-index $\alpha 
= (\alpha_{1}, \ldots, \alpha_{N})  \in \mathbb{Z}_{+}^{N}$, let $x^{\alpha} = 
x_{1}^{\alpha_{1}} \cdots x_{N}^{\alpha_{N}}$, $D^{\alpha} 
= D_{1}^{\alpha_{1}} \cdots
D_{N}^{\alpha_{N}}$. $\nabla_{l} =  (D^{\alpha})_{|\alpha|=l}$ 
denotes  the gradient of order 
$l$, where $l$ is a fixed nonnegative integer.  Denoting $m = m (N, l)$   to be 
the number of all 
multi-indices $\alpha = (\alpha_{1}, \ldots, \alpha_{N})$  such that $| \alpha | 
= \alpha_{1} 
+ \cdots + \alpha_{N} = l,$ $\nabla_{l}$  can be  viewed as an operator acting 
from 
$L_{2} (\Omega)$ into $L_{2} (\Omega; \mathbb{C}^{m})$  defined on its 
maximal domain, the Sobolev space $W_{2}^{l} (\Omega)$,  by 
\begin{equation*}\nabla_{l} u = (D^{\alpha} u)_{| \alpha | = l},\quad \ u \in  
W_{2}^{l} (\Omega).\end{equation*}
Recall that the Sobolev space $W_{2}^{l} (\Omega)$  consists of those 
functions $u \in  
L_{2} (\Omega)$ whose distributional derivatives $D^{\alpha} u$  belong to 
$L_{2} 
(\Omega)$    for all $\alpha \in \mathbb{Z}_{+}^{N},  | \alpha | \leq l.$  Equipped 
with the norm 
\begin{equation}\label{e:normawdoiel}\| u \|_{W_{2}^{l} (\Omega)} = 
\biggl( \sum_{|\alpha|\leq m} \|D^\alpha 
u \|_{L_{2} (\Omega)}^2 \biggl)^{1/2},\end{equation}
$W_{2}^{l} (\Omega)$  becomes a Hilbert space that is continuously embedded in 
$L_{2} (\Omega)$.    Also, recall that $\stackrel
\circ{W}_{2}^{l} (\Omega)$ denotes  the closure of $C_{0}^{\infty} (\Omega)$ in 
the space $W_{2}^{l} (\Omega)$. 
Besides, we will use the spaces $\stackrel \circ{L}_{p}^{l} (\Omega)$ 
(for $p = 1, 2$).   
The space $\stackrel \circ{L}_{p}^{l} (\Omega)$, $(1 \leq p < \infty)$ is defined 
as the completion of $C_{0}^{\infty} (\Omega)$ under the metric corresponding 
to 
\begin{equation*}\| u \|_{p, l} :=   \| \nabla_{l} u \|_{L_{p} (\Omega)} = 
\biggl( \int_{\Omega} 
\biggl( \sum_{|\alpha| = l} | D^\alpha u (x) |^{2}  \biggl)^{p/2}  \de x \biggl)^{1/p}, 
\quad u \in 
C_{0}^{\infty} (\Omega).\end{equation*}
The elements of  $\stackrel \circ{L}_{p}^{l} (\Omega)$ can be realized as locally 
integrable functions on $\Omega$ vanishing at the boundary $\partial \Omega$ 
and having distributional  derivatives of order $l$ in $L_{p} (\Omega)$. 
Moreover, these 
functions, after modification on a set  of zero measure, are absolutely 
continuous on 
every line which is parallel to the  coordinate axes, cf.\ 
O.~Nikodym~\cite{Nikodym}, 
S.M.~Nikolski \cite{Nikolski} (see also V.M.~Maz'ja \cite{Mazja}).
 
Further, suppose that on $\Omega$ there is defined an $m \times m$ matrix
valued measurable function $a$, more precisely, 
$a(x) = [a_{\alpha \beta} (x)]$, $|\alpha|,  |\beta| = l$,
$x\in\Omega$, where the scalar valued functions $a_{\alpha,\beta}$ are  
measurable on $\Omega$ for all multi-indices $|\alpha|, |\beta| = l$.  
We impose the following conditions.\smallskip

(C1)  {\it For almost all (with respect to the $n$-dimensional  standard 
Lebesgue measure) 
$x\in \Omega$,  the matrix $a (x)$  is nonnegative (positive semidefinite), 
that is,
\begin{equation*} \sum_{|\alpha|, |\beta| = l}a_{\alpha \beta}(x) 
\overline{\eta}_\beta \eta_\alpha \geq 0,\mbox{ for all }\eta=(\eta_\alpha)_{|\alpha| =l}\in 
\mathbb{C}^m.
\end{equation*}}

According to the condition (C1), there exists an $m \times m$  matrix valued 
measurable function $b$ on $\Omega$, such that 
\begin{equation*}a (x) = b (x)^{*} b (x), \mbox{ for almost all }  x \in \Omega,
\end{equation*}
where $b(x)^*$ denotes the Hermitian conjugate matrix of the matrix $b(x)$.
Here and hereafter, it is convenient to consider $m \times m$ matrices as 
linear 
transformations in $\mathbb{C}^{m}$. Also, $|\cdot|$ denotes the unitary 
norm (the $\ell_2$ norm) in $\mathbb{C}^m$.\smallskip

(C2) {\it There is a nonnegative measurable function $c$ on $\Omega$ 
such that, for 
almost all $x\in\Omega$ and all $\xi=(\xi_1,\ldots,\xi_N)\in\CC^N$,
\begin{equation*}|b (x) \widetilde{\xi}| \geq  c (x) |\tilde{\xi}|,\end{equation*}
where $\tilde{\xi}  = (\xi^{\alpha})_{|\alpha| = l}$ is the vector in 
$\mathbb{C}^{m}$ with  
$\xi^{\alpha} = \xi_{1}^{\alpha_{1}} ... \ \xi_{N}^{\alpha_{N}}$}.\smallskip

(C3) {\it All the entries $b_{\alpha \beta}$  of the $m\times m$ 
matrix valued function $b$ are functions in $L_{1, \mathrm{loc}}(\Omega)$.}\smallskip

(C4) {\it The function $c $  in \emph{(C2)} has the property that $1 \big/ c 
\in L_{2} (\Omega)$.}\smallskip

Under the conditions (C1)--(C4), we consider  the operator $T$ acting from  
$L_{2} (\Omega)$ to $L_{2} (\Omega; \mathbb{C}^{m})$ and defined by
\begin{equation}\label{e:te}
(T u) (x) = b (x) \nabla_{l} u (x), \quad\mbox{ for almost all } 
x \in \Omega,
\end{equation}
on its domain
\begin{equation}\label{e:domte}
\dom (T) = \{ u \in  \stackrel \circ{W}_{2}^{l} (\Omega) \mid b 
\nabla_{l} u  \in L_{2} (\Omega; \mathbb{C}^{m}) \}.\end{equation}

Our aim is to describe, in view of the abstract model proposed in
Section~\ref{s:model}, the triplet of closely embedded Hilbert spaces 
$(\mathcal{D}(T);L_{2}(\Omega);\mathcal{R} (T^{*}))$ associated with the 
operator $T$ defined at \eqref{e:te} and \eqref{e:domte}. In terms of these results, we
obtain 
information about weak solutions 
for  the corresponding operator equation involving the Hamiltonian operator  
$H = T^{*} T$ of the triplet, which in fact is a  Dirichlet boundary value problem 
in $L_{2} (\Omega)$ with homogeneous boundary values. 
This problem is associated  to the differential sesqui-linear form 
\begin{align}\label{e:forma}
a [u, v] & = \ \int_{\Omega} \langle a (x) \nabla_{l} (x) , \nabla_{l} (x) \rangle 
\de x
\\ &  =  \sum_{| \alpha | = | \beta | = l} \ \int_{\Omega} a_{\alpha \beta} (x) D^{\beta} u (x) 
\overline{D^{\alpha} v (x)} d x, \quad u, v \in  C_{0}^{\infty} (\Omega),\nonumber
\end{align}
which, as will be seen, can be extended up to elements of $\mathcal{D} (T)$.  The 
problem can be   reformulated as follows : {\it given $f \in \mathcal{D} (T)^{*}$ 
(which is canonically identified withe $\mathcal{R} (T^{*}))$,  
find $v  \in \mathcal{D} (T)$  such that} 
\begin{equation}\label{e:formadoi}
a [u, v] = \ \langle u, f \rangle \mbox{ for all } u \in \cD (T),
\end{equation}
where $\langle \cdot, \cdot \rangle$ denotes the duality between 
$\mathcal{D} (T)$  and $\mathcal{D} (T)^{*}$. 
The problem in \eqref{e:formadoi} can be considered only for $ u \in \dom (T)$,  
or, even more restrictively, only for  $u \in  C_{0}^{\infty} (\Omega)$. 

We first prove a useful inequality.

\begin{lemma}\label{inequa}  
Under the conditions \emph{(C1)} through \emph{(C4)}, there holds the inequality 
\begin{equation}\label{inequal}
\int_{\Omega} |\nabla_{l} u (x) | \de x  \leq \ C \biggl( \int_{\Omega} |b (x) \nabla_{l} u (x) |^{2} 
\de x \biggl)^{\frac{1}{2}}, \quad u \in C_{0}^{\infty} (\Omega),\end{equation}
where 
\begin{equation*}C =  \biggl( \int_{\Omega} c (x)^{- 2} \de x \biggl)^{\frac{1}{2}}.
\end{equation*}
\end{lemma}

\begin{proof} For any function $u \in C_{0}^{\infty} (\Omega)$, due to condition (C2), 
we have 
\begin{equation*}|b (x) \nabla_{l} u  (x)| \geq  c (x) |\nabla_{l} u  (x)|, \mbox{ 
for almost all }
 x \in \Omega.\end{equation*}
Hence
\begin{equation*}|\nabla_{l} u  (x)| \leq c (x)^{-1} | b (x) \nabla_{l} u  (x)|,  \mbox{ for almost all }
 x \in \Omega, \end{equation*}
and then, integrating over $\Omega$ and then using Schwarz inequality, we obtain 
\begin{align*}\int_{\Omega} \mid  \nabla_{l} u (x) \mid d x  & \leq \int_{\Omega} 
c (x)^{- 1} | b (x) \nabla_{l} u (x)|   \de x  \\
& \leq  \biggl( \int_{\Omega} c(x)^{- 2} d x \biggl)^{\frac{1}{2}}  
\biggl( \int_{\Omega} | b (x) \nabla_{l} u (x) |^{2} \de x \biggl)^{\frac{1}{2}},\end{align*}
hence the inequality \eqref{inequal}.\end{proof}
 
Secondly we investigate the topological properties of the operator $T$.
 
\begin{lemma}\label{l:cdi} 
Under the conditions \emph{(C1)--(C4)}, the operator $T$, 
defined at \eqref{e:te} and \eqref{e:domte}, is closed, densely 
defined, and injective.\end{lemma}

\begin{proof} By (C3), all entries 
$b_{\alpha\beta}$ of $b$ are functions in $L_{1,\mathrm{loc}} (\Omega)$, 
therefore $C_0^\infty (\Omega) \subseteq \dom (T)$, hence
$T$ is densely defined. The injectivity of  $T$ 
follows from the inequality \eqref{inequal} given in Lemma \ref{inequa}. 

In order to prove that $T$ is closed, let 
$(u_{n})$ be a sequence of  elements $u_{n} \in \dom(T)$, i.e.\ 
$u_n \in \stackrel\circ{W}_{2}^{l} (\Omega)$  for which 
$b  \nabla_{l} u_n \in L_{2} (\Omega; \mathbb{C}^{m})$,
such that $T u_n \rightarrow v$ in the norm  of $L_{2} (\Omega; \mathbb{C}^{m})$ 
and $u_{n} \rightarrow u$ in the norm of $L_{2} (\Omega)$. It follows that 
\begin{equation*}
\int_{\Omega} | b  \nabla_{l} (u_{n} - u_{k}) |^{2} \de x \rightarrow 
0,\mbox{ as }n,k\rightarrow \infty\end{equation*} and, by Lemma~\ref{inequa},
\begin{equation*}
\int_{\Omega} |\nabla_{l} (u_{n} - u_{k})| \de x \rightarrow 0,\mbox{ as }
n,k\rightarrow\infty\end{equation*}
that is,
\begin{equation*}\| u_{n} - u_{k} \|_{\stackrel\circ{L}_{1}^{l} (\Omega)} \rightarrow 0 
\mbox{ as } n, k \rightarrow \infty.\end{equation*}

Since $\stackrel\circ{L}_{1}^{l} (\Omega)$  is a complete space and 
the gradient of functions  in $\stackrel\circ{L}_{1}^{l} (\Omega)$, considered 
in the sense of distributions, coincides almost everywhere with the gradient 
considered in the sense of ordinary derivatives, see Theorem 1.1.3/1 
  in V.G.~Maz'ja \cite{Mazja}, it follows that there is an element 
$\widetilde{u} \in \stackrel\circ{L}_{1}^{l} (\Omega)$ such that 
\begin{equation*}\int_{\Omega} |\nabla_{l} (u_{n} - \tilde{u})| \de x \rightarrow 0 
\mbox{ as } n  \rightarrow \infty.\end{equation*}
Note also that, without loss of generality, 
we can assume that $u_{n} \rightarrow  u$  pointwise almost everywhere on 
$\Omega$: otherwise, we may use a subsequence of $(u_{n})$.

For any $\varphi \in C_{0}^{\infty} (\Omega)$ we have
$u \overline{\nabla_{l}^{*} \varphi}  \in L_{1} (\Omega)$,
and then by the Dominated Convergence Theorem of Lebesgue, one gets
\begin{align*}\langle u, \nabla_{l}^{*} \varphi \rangle_{L_{2} (\Omega)} & 
=  \int_{\Omega} u  \overline{\nabla_{l}^{*} \varphi} \de x \\
&= \lim_{n \rightarrow \infty} \int_{\Omega} u_{n} \overline{\nabla_{l}^{*} \varphi} \de x  
= \lim_{n \rightarrow \infty} \int_{\Omega}  \langle \nabla_{l} u_{n}, \varphi \rangle\de x\\ 
&=  \int_{\Omega} \langle \nabla_{l} \widetilde{u}, \varphi \rangle \de x.\end{align*}
Therefore, $u \in \dom (\nabla_{l})$ and $\nabla_{l} u  = \nabla_{l} \widetilde{u}$,  
hence
\begin{equation*}\int_{\Omega} \ \mid \nabla_{l} (u_{n} - u) \mid \de x \rightarrow 
0\mbox{ as } n  \rightarrow \infty.\end{equation*}
Moreover,
\begin{equation*}\| u - \widetilde{u} \|_{\stackrel\circ{L}_{1}^{l} (\Omega)}\leq 
\| u - u_n\|_{\stackrel\circ{L}_{1}^{l} (\Omega)}  +  \| u_n - \widetilde{u} \|_{\stackrel
\circ{L}_{1}^{l} (\Omega)}  \rightarrow 0,\end{equation*}
so $ u = \widetilde{u} \in \stackrel\circ{L}_{1}^{l} (\Omega)$.

Also, we have
\begin{equation*}\lim_{n \rightarrow \infty} \nabla_{l} u_{n} (x) =  \nabla_{l} u (x),  
 \quad   \mbox{ for almost all }x \in \Omega \end{equation*}
Then, by Fatou's Lemma,  
\begin{equation*}\int_{\Omega} |b  \nabla_{l} (u_{n} - u)|^{2} \leq \liminf_{k} 
\int_{\Omega} | b  \nabla_{l} (u_{n} - u_k) |^{2} d x \leq \epsilon.\end{equation*}
It follows that
\begin{equation*} b  \nabla_{l} u =  b  \nabla_{l} (u - u_{n}) + b  \nabla_{l} u_{n}  
\in L_{2} (\Omega; \mathbb{C}^{m})\end{equation*}     
and 
\begin{equation*}\int_\Omega |b  \nabla_l (u_{n} - u) |^{2} \rightarrow 0 \mbox{ as } n 
\rightarrow \infty.\end{equation*}  
Therefore $u \in \dom (T)$,   $v = b \nabla_{l}  u$, i.e.\ $v =Tu$, and
the closedness of $T$ is proven.\end{proof}    

As a consequence of Lemma~\ref{l:cdi}, we can now apply Theorem~\ref{t:model} 
and the underlying constructions to the operator $T$.  To this 
end, it will be convenient to consider $T$ as an operator acting from 
$L_2 (\Omega)$ to the space obtained by the  closure of $\ran (T)$ in 
$L_2 (\Omega; \mathbb{C}^m)$.  
Obviously, all properties in the previous lemma remain true for this restriction 
as well. We now follow the model space $\cD(T)$ as in Subsection~\ref{ss:sdt} 
and define 
\begin{equation}\label{e:modte}|u |_{T} : =  \biggl( \int_{\Omega} |b (x)  \nabla_{l} u  (x) |^{2} \de x 
\biggl)^{\frac{1}{2}}, \ u \in\dom(T).\end{equation}
Recall that $b$ is determined by $a (x) = b^{*} (x) b (x)$  for almost all $x \in \Omega$
and note that, due to the conditions (C2) through (C4), this is a
pre-Hilbert norm on $\dom (T)$.  The corresponding inner product is 
given by, 
\begin{equation}\label{e:psmodte} ( u, v )_{T} = \int_{\Omega} \langle b (x) \nabla_{l} u (x),  b (x) 
\nabla_{l} v (x) \rangle \de x \end{equation}
for $u,v\in\dom(T)$. Let $\mathcal{D} (T)$ denote the Hilbert space obtained 
by an abstract completion of $\dom (T)$  with respect to the norm 
$| \cdot |_{T}$ defined at \eqref{e:modte}.  
In order to use efficiently this space, we have to choose a special representation
of the space  $\mathcal{D} (T)$ that can be realized inside
the space $\stackrel\circ{L}_{1}^{l} (\Omega)$,  with elements 
functions on $\Omega$.

\begin{proposition}\label{p:hazeroreal}
The Hilbert space  $\mathcal{D} (T)$  has a realization that 
is continuously embedded in $\stackrel\circ{L}_{1}^{l} (\Omega)$.
\end{proposition}

 \begin{proof} Let $u$ be an arbitrary element of the space  $\mathcal{D} (T)$.  Then, there exists a sequence $(u_{n})$, with 
 all elements in $\dom(T)$,  such that 
 \begin{equation*}|u_{n} - u |_{T}  \rightarrow 0 \mbox{ as } n \rightarrow 
 \infty.\end{equation*}
 In particular,
\begin{equation*} |u_{n} - u_{k} |_{T}^{2} =   \int_{\Omega} | b (x) (\nabla_{l} 
(u_n-u_k) (x)|
^{2} \de x \rightarrow 0 \mbox{ as }  n, k \rightarrow \infty.    \end{equation*}
 In view of the inequality in Lemma~\ref{inequa}, 
it follows that 
\begin{equation*}\| u_{n} - u_{k} \|_{\stackrel\circ{L}_{1}^{l} (\Omega)}\rightarrow 0 
\mbox{ as } n, k \rightarrow \infty.\end{equation*}

Since $\stackrel\circ{L}_{1}^l (\Omega)$ is complete there exists a function  
$v \in  \stackrel\circ{L}_{1}^{l} (\Omega)$  such that 
\begin{equation*}\| u_{n} - v \|_{\stackrel\circ{L}_{1}^{l} (\Omega)} \ \rightarrow 0 
\mbox{ as } n \rightarrow \infty.\end{equation*}
The element $v$ depends only on $u$, more precisely, it is not depending on the 
chosen sequence $(u_{n})$. Therefore,
it can be defined an operator 
$J_{a}  \colon  \mathcal{D} (T)   \rightarrow \ \stackrel
\circ{L}_{1}^{l} (\Omega)$  by setting
\begin{equation*}J_{a} u = v, \quad u \in   \mathcal{D} (T).\end{equation*}
$J_{a}$ is an injective operator. To see this, if $J_{a} u = 0$, then for a suitable 
sequence $(u_{n})$,  $u_{n} \in \dom (T)$, we have $\mid u_{n} - u \mid_{T}^{2} 
\rightarrow 0$, and
\begin{equation*}\|u_{n} \|_{\stackrel\circ{L}_{1}^{l} (\Omega)} \sim   \int_{\Omega} |
\nabla_{l} u_{n}  (x)  | \de x  \rightarrow 0 \mbox{  as } n \rightarrow \infty.
\end{equation*}
It can be assumed that  $\nabla_{l} u_{n} \rightarrow 0$ almost  everywhere,
otherwise, we may pass to a subsequence of $(u_{n})$.  For any $\epsilon > 0$  
and sufficiently large $n$ and $k$, there holds
\begin{equation*} |u_{n} - u_{k} |_{T}^{2} = \int_{\Omega} | b  \nabla_{l} (u_{n} - 
u_{k} ) |^{2} \de x < \epsilon,\end{equation*} 
and, by applying Fatou's Lemma,
\begin{align*}\| u_{n}\|_{T}^{2} &= \int_{\Omega} | b  \nabla_{l} u_{n}|^{2} \de x =  
\int_{\Omega} \lim_{k\rightarrow \infty} | b  \nabla_{l}  (u_{n} - u_{k})  |^{2} \de x \\
& \leq  \liminf_{k}   \int_{\Omega} \ \mid b  \nabla_{l} (u_{n} - u_{k}) \mid^{2} \de x   
\leq  \epsilon.\end{align*}
Thus, $u_{n} \rightarrow 0$  in $\mathcal{D} (T)$ and 
hence $u = 0$. We conclude that the operator $J_{a}$ is injective, therefore 
the space $\mathcal{D} (T)$ can be realized by means of 
functions in $\stackrel\circ{L}_{1}^{l} (\Omega)$.
Moreover, the embedding of $\mathcal{D} (T)$ 
into $\stackrel\circ{L}_{1}^{l} (\Omega)$ is continuous, that again is a 
consequence of the inequality in 
Lemma~\ref{inequa}  which, obviously, can be extended for all
$u\in  \mathcal{D} (T)$. 
\end{proof}

As a consequence of Proposition \ref{p:hazeroreal}, let 
$\stackrel\circ{\mathcal{H}}_{a}^{l}  (\Omega)$ denote
the concrete realization  $\mathcal{D} (T)$  continuously embedded into 
$\stackrel\circ{L}_{1}^{l} (\Omega)$. 
Moreover, according to the assertions in items (i) and (ii) of 
Theorem~\ref{t:model}, this space $\stackrel
\circ{\mathcal{H}}_{a}^{l}  (\Omega)$  is closely  embedded in 
$L_{2} (\Omega)$       
and, in turn, $L_{2} (\Omega)$  is closely embedded in the conjugate space  
$\bigl( \stackrel\circ{\mathcal{H}}_{a}^{l}  (a) \bigl)^{*}$, that we denote by  
$\stackrel\circ{\mathcal{H}}_{a}^{-l}  (\Omega)$. Moreover, 
$( \stackrel\circ{\mathcal{H}}_{a}^{l}  
(\Omega); \ L_{2} (\Omega); \ \stackrel\circ{\mathcal{H}}_{a}^{-l}  (\Omega) )$
is a triplet of closely embedded Hilbert spaces in the sense of the definition as 
in Subsection~\ref{ss:dbp}. 
 
Further on,
by Theorem \ref{t:model} (vi), the conjugate space of 
$\stackrel\circ{\mathcal{H}}_{a}^{l}  (\Omega)$,  that is, 
$\stackrel\circ{\mathcal{H}}_{a}^{-l}  (\Omega)$, is canonically identified with 
$\mathcal{R} (T^{*})$.  
In general,  this is not a space of distributions on $\Omega$. 
On the other hand, 
for every $f \in (\stackrel\circ{\mathcal{H}}_{a}^{l}  (\Omega))^{*}$ there 
exist elements $g \in L_{2} (\Omega; \mathbb{C}^{m})$  such 
that \begin{equation}\label{e:fu}
f (u) = \int_{\Omega} \langle g(x), b (x) \nabla_{l} u (x) \rangle \de x, \quad u 
\in  \stackrel\circ{W}_{2}^{l} (\Omega).
\end{equation}
Moreover,
\begin{equation*}\| f \|_{(\stackrel\circ{\mathcal{H}}_{a}^{l}  (\Omega))^{*}} = 
\inf\{ \| g \|_{L_{2} (\Omega; \mathbb{C}^{m})} \mid g \in {L}_{2} 
(\Omega;\CC^m) \mbox{ such that \eqref{e:fu} holds }\} .\end{equation*}
 
The Hamiltonian $H = T^{*} T$  of the   triplet can be viewed as an  operator 
associated with the differential sesqui-linear form $a$ defined as 
in \eqref{e:forma}. We recall 
that $a$, on $C_{0}^{\infty} (\Omega)$,  coincides with  the inner product $(\cdot, 
\cdot)_{T}$, and hence $a$ can be extended on $\mathcal{D} (T)$  by 
\begin{equation*}a [u, v] = (u, v)_{T}, \quad u, v \in \mathcal{D} (T).\end{equation*}
On the other hand, due to Theorem \ref{t:model} (iv), 
$H$ admits an extension  to a unitary 
operator $\widetilde{H}$ acting between  $\stackrel\circ{\mathcal{H}}_{a}^{l}  (\Omega)
$ and $\stackrel\circ{\mathcal{H}}_{a}^{-l}  (\Omega)$. Therefore, the form $a$  
extended on $\stackrel\circ{\mathcal{H}}_{a}^{l}  (\Omega)$,   is associated with $
\widetilde{H}$. Consequently, the problem defined by \eqref{e:formadoi} 
is equivalent with the operator  equation
\begin{equation}\label{e:has}
\widetilde{H} v = f, \quad f \in  \stackrel\circ{\mathcal{H}}_{a}^{-l}  (\Omega).
\end{equation}
Thus, a solution of \eqref{e:forma} is treated as a weak 
solution for \eqref{e:has}. It is sufficient to verify \eqref{e:has}  for 
$u \in \ \stackrel\circ{W}_{a}^{l}  (\Omega)$  or  on another dense subspace 
in  $\stackrel\circ{\mathcal{H}}_{a}^{l}  (\Omega)$ as, for instance, 
$C_{0}^{\infty} (\Omega)$.
  
The preceeding considerations can be summarized in the following

\begin{theorem}\label{t:weaksolutions}
For $\Omega$ a domain in $\RR^N$ and $l\in\NN$, let
$a(x) = [a_{\alpha \beta} (x)]=b(x)^*b(x)$, $|\alpha|,  |\beta| = l$,
$x\in\Omega$, satisfy the conditions \emph{(C1)--(C4)}, and consider
the differential sesqui-linear form 
\begin{align*}\label{e:formaprim}
a [u, v] & = \ \int_{\Omega} \langle a (x) \nabla_{l} (x) , \nabla_{l} (x) \rangle 
\de x
 =  \sum_{| \alpha | = | \beta | = l} \ \int_{\Omega} a_{\alpha \beta} (x) 
D^{\beta} u (x) 
\overline{D^{\alpha} v (x)} d x, \quad u, v \in  C_{0}^{\infty} (\Omega),
\nonumber
\end{align*} Then:

\nr{1} The operator $T$ acting from  
$L_{2} (\Omega)$ to $L_{2} (\Omega; \mathbb{C}^{m})$ and defined by
$(T u) (x) = b (x) \nabla_{l} u (x)$ for 
$x \in \Omega$ and $u\in
\dom (T) = \{ u \in  \stackrel \circ{W}_{2}^{l} (\Omega) \mid b 
\nabla_{l} u  \in L_{2} (\Omega; \mathbb{C}^{m}) \}$ is closed, densely defined, 
and injective.

\nr{2} The pre-Hilbert space $\dom(T)$ with norm 
$|u |_{T} =  ( \int_{\Omega} |b (x)  \nabla_{l} u  (x) |^{2} \de x 
)^{\frac{1}{2}}$, has a unique Hilbert space completion, denoted
by  $\mathcal{H}_{a}^{l}  (\Omega)$, that is 
continuously embedded into $\stackrel\circ{L}_{1}^{l} (\Omega)$.

\nr{3} The conjugate space of $\stackrel\circ{\mathcal{H}}_{a}^{l}\!\! (\Omega)$,
denoted by $\stackrel\circ{\mathcal{H}}_{a}^{-l}\!\!(\Omega)$, can be realized 
in such a way that, for any $f\in \stackrel\circ{\mathcal{H}}_{a}^{-l}  (\Omega)$ 
there exist elements $g \in L_{2} (\Omega; \mathbb{C}^{m})$  such 
that \begin{equation}\label{e:futrei}
f (u) = \int_{\Omega} \langle g(x), b (x) \nabla_{l} u (x) \rangle \de x, \quad u 
\in  \stackrel\circ{W}_{2}^{l} (\Omega),
\end{equation}
and
\begin{equation*}\| f \|_{\stackrel\circ{\mathcal{H}}_{a}^{-l}\!\!(\Omega)} = 
\inf\{ \| g \|_{L_{2} (\Omega; \mathbb{C}^{m})} \mid g \in {L}_{2}(\Omega;\CC^m) \mbox{ such that \eqref{e:futrei} holds }\} .\end{equation*}

\nr{4} $(\stackrel\circ{\mathcal{H}}_{a}^{l}\!\!(\Omega); L_2(\Omega);\stackrel
\circ{\mathcal{H}}_{a}^{-l}\!\!(\Omega))$ is a triplet of closely embedded Hilbert
spaces.

\nr{5} For every $f \in \stackrel
\circ{\mathcal{H}}_{a}^{-l}\!\! (\Omega)$ there exists a unique 
$v \in \mathcal{H}_{a}^{l}  (\Omega)$ that solves the Dirichlet problem 
associated to the sesquilinear form $a$, in the sense that
\begin{equation*}
a [u, v] = \langle u, f \rangle \mbox{ for  all } u \in  \mathcal{H}_{a}^{l}  (\Omega).
\end{equation*}
More precisely, $v=\widetilde H^{-1}f$, where $\widetilde H$ is the unitary 
operator acting between  $\stackrel\circ{\mathcal{H}}_{a}^{l}  (\Omega)
$ and $\stackrel\circ{\mathcal{H}}_{a}^{-l}  (\Omega)$ that uniquely extends 
the positive selfadjoint operator $H=T^*T$ in $L_2(\Omega)$.
\end{theorem}  
  
  \begin{remark}
Since the Hamiltonian $H$ is associated to a 
differential form, it can be treated as a formal  differential operator
\begin{equation}\label{e:hus}
H u = \sum_{|\alpha| = l}  \sum_{|\beta| = l} D^{\alpha}  (a_{\alpha, \beta}  (x) 
D^{\beta} u).
\end{equation}  
However, it should be emphasized that  $H$ in \eqref{e:hus} should be 
rather considered a symbol that may not be 
a differential  operator at all, due to the fact that 
the coefficients $a_{\alpha, \beta}$ are not assumed 
to be differentiable. If we impose conditions of smoothness on the boundary of 
$\Omega$ and on $a_{\alpha,\beta}$ then $H$ in \eqref{e:hus} may be a 
differential operator.
\end{remark}

\begin{remark}
In case $1/c  \in L_{\infty} (\Omega)$  the following inequality can be proved 
\begin{equation*}\int_{\Omega} | \nabla_{l} u (x) |^{2} d x \leq C \int_{\Omega} 
| b(x) \nabla_{l} u (x) 
|^{2} \de x, \quad u \in C_{0}^{\infty} (\Omega),\end{equation*}
where $C$ is a constant that is independent of $u$. 
In this case, with arguments
similar to those used in the proof Proposition~\ref{p:hazeroreal}, $\stackrel
\circ{\mathcal{H}}_{a}^{l}  (\Omega)$ is a space of functions that admits 
a natural 
continuous embedding into $\stackrel\circ{L}_{2}^{l} (\Omega)$.  For 
a bounded domain $\Omega$, due to the Poincar\'e Inequality 
\begin{equation*}
\|u\|_{L_{2} (\Omega)} \leq  c \|u\|_{\alpha, \beta}, \quad u \in C_{0}^{\infty} 
(\Omega),\end{equation*}
the norm $\| \cdot \|_{2, l}$ is equivalent to the Sobolev norm $\| \cdot \|_{W_{2}^{l} 
(\Omega)}$.  It follows $\stackrel\circ{L}_{2}^{l} (\Omega) = \stackrel\circ{W}_{2}^{l} 
(\Omega)$, the space $\stackrel\circ{\mathcal{H}}_{a}^{l}  (\Omega)$ is realized as a 
subspace of $\stackrel\circ{W}_{2}^{l} (\Omega)$, and, in this case,  $\stackrel
\circ{\mathcal{H}}_{a}^{l}  (\Omega)$ is continuously embedded in 
$L_{2} (\Omega)$.
\end{remark}


\end{document}